\documentclass{amsart}

\usepackage[utf8]{inputenc} % set input encoding (not needed with XeLaTeX)

\usepackage[foot]{amsaddr} %package that puts the authors addresses on the title page as a footnote (when using AMSART document class only).

\usepackage[english]{babel}
\usepackage{enumerate}
\usepackage{amsmath,amsthm,amsfonts,amssymb,euscript,verbatim,bbm}
\usepackage[backref=none]{hyperref} %% This provides \href command
\usepackage[usenames,dvipsnames]{xcolor} % Driver-independent color extensions for LATEX and pdfLATEX
\definecolor{MyDarkBlue}{rgb}{0,0.1,0.7}
\hypersetup{pdfborder={0 0 0},colorlinks,breaklinks=true,
  urlcolor={MyDarkBlue},citecolor={MyDarkBlue},linkcolor={MyDarkBlue}
}
\usepackage{color}

\theoremstyle{plain}
\newtheorem{theorem}{Theorem}[section]
\newtheorem{lemma}[theorem]{Lemma}
\newtheorem{definition}[theorem]{Definition}
\newtheorem{proposition}[theorem]{Proposition}
\newtheorem{remark}[theorem]{Remark}

\numberwithin{equation}{section}
\numberwithin{theorem}{section} 
\DeclareMathOperator*{\esssup}{ess\,sup}
\newcommand{\eqdef}{\overset{\mbox{\tiny{def}}}{=}}

\newcommand{\pv}{p}
\newcommand{\qv}{q}
\newcommand{\pZ}{\pv^0}
\newcommand{\qZ}{\qv^0}

\newcommand{\threed}{{{\mathbb R}^3}}

\newcommand{\ud}{\,\mathrm{d}}

\newcommand{\eps}{\varepsilon}
\newcommand{\R}{{\mathbb{R}}}
\newcommand{\PP}{{\mathcal{P}}}
\newcommand{\IN}{{\textbf{1}}}

\newcommand{\ProbJ}{\mathcal{H}}

\title[Uniqueness for the relativistic Landau Equation]{Uniqueness of Bounded Solutions for the Homogeneous Relativistic Landau Equation with Coulomb Interactions}

\author[R. M. Strain]{Robert M. Strain$^\dagger$}
\address{$^\dagger$Department of Mathematics, University of Pennsylvania, Philadelphia, PA 19104, USA.  \href{mailto:strain@math.upenn.edu}{strain@math.upenn.edu}}
\thanks{$^\dagger$R.M.S. was partially supported by the NSF grants DMS-1500916 and DMS-1764177.}

\author[Z. Wang]{Zhenfu Wang$^\ddagger$}
\address{$^\ddagger$Department of Mathematics, University of Pennsylvania, Philadelphia, PA 19104, USA.  \href{mailto:zwang423@math.upenn.edu}{zwang423@math.upenn.edu}}


\date{}

\dedicatory{Dedicated to Professor Walter Strauss on the occasion of his eightieth birthday}


\begin{document}

%\date{\today}
%\date{\today; \Red{(DRAFT)}}
% \date{May 8, 2019}
% \date{}

%\let\thefootnote\relax\footnotetext{2010 \textit{Mathematics Subject Classification.} Primary: 82D10, 35Q70, 35Q75, 35B45, 35A02, 35Q70, 35A02. \\ %; Secondary: 83A05.\\
%	\textit{Key words and phrases.}  Relativistic Landau equation; Weak solutions; Stochastic representation; Uniqueness; Wasserstein distance.  }
%\addtocounter{footnote}{-1}\let\thefootnote\svthefootnote

\begin{abstract} 
We prove the uniqueness of weak solutions to the spatially homogeneous special relativistic Landau equation under the conditional assumption that the solution satisfies $(\pZ)^7 F(t,p) \in L^1 ([0,T]; L^\infty)$.  The existence of standard weak solutions to the relativistic Landau equation has been shown recently in \cite{StrainTas}.
\end{abstract}

\setcounter{tocdepth}{1}
%\chapter is level 0
%\section is level 1
%\subsection is level 2
%\subsubsection is level 3
%\paragraph is level 4
%\subparagraph is level 5
\maketitle

%\bibliographystyle{amsplain}

%\textbf{Mathematics Subject Classification (2010).} Primary: 82D10, 35Q70, 35Q75, 35B45, 35A02, 35Q70, 35A02. \\ % Secondary:\\

%    \textbf{Keywords}: Relativistic Landau equation; Weak solutions; Stochastic representation; Uniqueness; Wasserstein distance.  

\tableofcontents 

\section{Introduction}

\thispagestyle{empty}

In this article we study the spatially homogeneous special relativistic Landau Equation with Coulomb interactions which is a basic model in Kinetic theory.   The Boltzmann equation is perhaps the most widely used partial differential equation in Kinetic theory.  However the Boltzmann equation does not make sense for the important Coulomb interactions \cite{MR1650006}. 
In 1936, Landau introduced a correction to the Boltzmann equation that is generally used to model a dilute hot plasma where fast moving particles interact via Coulomb interactions \cite{HintonArticle,MR684990}.  This partial differential equation, which is now called the Landau equation, does not include the effects of Einstein's theory of special relativity. However for  particle velocities that are close to the speed of light, which occurs commonly in a hot plasma,  relativistic effects are very important. The relativistic version of Landau's equation was derived by Budker and Beliaev in 1956 \cite{MR0083886,BelyaevBudker}. It is a fundamental model for studying the dynamics of a dilute collisional plasma.

The relativistic Landau equation is given by
\begin{equation}
\label{RelLandau}
\partial_t F =  \mathcal{C}(F, F).
\end{equation} 
Here $F=F(t, p)$ is the density and $p \in \mathbb{R}^3$ is the momentum variable and $t\ge 0$ is the time variable.  This equation includes the initial conditions $F(0,p) = F_0(p)$.  The collision operator can be written as
\begin{equation}\label{ColOper}
\mathcal{C}(h, g)(p) = \frac{1}{2}\sum_{i, j=1}^3 \partial_{p_i}  \int_{\mathbb{R}^3} \ud q  ~ \Phi^{ij}(p, q)  \left[h(q) \partial_{p_j} g(p) - g(p) \partial_{q_j} h(q) \right]. 
\end{equation}
The relativistic Landau kernel $(\Phi^{ij}(p, q))_{1 \leq i, j \leq 3}$  is then given by  
\begin{equation}\label{ColKer}
\Phi^{ij} (p, q)=  \Lambda(p, q) S^{ij}(p, q).
\end{equation}
For the momentum $p, q \in \mathbb{R}^3$,  we set the energies to be $p^0 = \sqrt{1 + |p|^2}$ and $q^0=   \sqrt{1+ |q|^2}$.  Then the relativistic relative momentum is defined by
\begin{equation}\label{rho.def.eq}
\rho  \eqdef  p^0 q^0 - p \cdot q -1 = \frac{|p-q|^2 + |p \times q|^2}{p^0 q^0 + p \cdot q +1}\geq 0.
\end{equation}
The proof of this identity  is straightforward since in particular one can use the formula 
$
|p  \times q|^2 + |p \cdot q|^2 = |p|^2 |q|^2. 
$
Then  $\tau = \rho +2$ and we have that 
\begin{equation}\label{lambda.def}
\Lambda(p, q)= \frac{(\rho +1)^2}{p^0 q^0} (\rho \tau )^{- \frac{3}{2}},
\end{equation}
and
\begin{equation}\label{Smatrix.def}
S^{ij}(p, q)= \rho \tau  \delta_{ij} - (p_i -q_i)(p_j -q_j) + \rho(p_i q_j + p_j q_i). 
\end{equation}
In this formulation we can directly observe that the matrix of the relativistic Landau kernel, $\Phi$, has a first order non-isotropic singularity because of \eqref{rho.def.eq}.

The main point of this article is to prove a conditional uniqueness result for large data weak solutions to the relativsitic Landau equation \eqref{RelLandau}.  This is stated in Theorem \ref{MainTheorem} below.  To prove this theorem we introduce several new decompositions and perform challenging pointwise estimates for the relativistic Landau kernel \eqref{ColKer}; these estimates build upon  recent difficult algebraic estimates in \cite{StrainTas}.    We also introduce  a stochastic representation, in  \eqref{SDE1} and \eqref{SDE2}, of solutions to the relativsitic Landau equation \eqref{RelLandau} with the specific coefficient matrix $\Sigma$ that is introduced in Proposition \ref{CoeffBM}.  
The work of Tanaka \cite{MR512334} in 1978 used a stochastic approach and proved uniqueness for the Boltzmann equation without cutoff in the Maxwell molecules case. Further see \cite{MR791288}.      Our main results make use of the approach by Fournier and Gu\'erin in \cite{MR1970275, MR2398952,MR2502525, MR2718931}, which proved uniqueness by looking at the stochastic representation of the classical non-relativistic spatially  homogeneous Landau equation.  In particular our result can be seen as the special relativistic counterpart of the result in  Fournier \cite{MR2718931} which proved the uniqueness for bounded solutions of the non-relativistic spatially homogenenous Landau equation with Coulomb interactions.   We refer to \cite{StrainTas} for a recent comparison of the relativistic and non-relativistic Landau equations.

The major new difficulties in the proof our main Theorem \ref{MainTheorem} are largely algebraic, due to the extreme complexity of the structure of the relativistic Landau kernel.  In this paper we introduce the stochastic coefficient matrix \eqref{sigmaM.def} $\Sigma$ used in the stochastic differential equations \eqref{SDE1} and \eqref{SDE2}, and we further prove several detailed pointwise estimates of $\Sigma$ and other quantities in order to establish the uniqueness theorem.  

It is well known that relativistic Kinetic theory contains many extreme difficulties at the level of the algebraic structure of the collision kernels, due to the quantities that arise in special relativity.  We note that an extensive study of the pointwise behavior of the collision operators in relativistic Kinetic theory was done by Glassey and Strauss between 1991-1995 in their work on the relativistic Boltzmann equation; see \cite{MR1321370,MR1211782,MR1392991,MR1275397,MR1105532}.  In particular \cite{MR1105532} gives an understanding of the very complex Jacobian of the pre-post change of variables for the relativistic Boltzmann equation.  Then \cite{MR1211782} proves the global asymptotic stability and uniqueness of the relativistic Boltzmann equation in the Torus.  Afterwards \cite{MR1321370} generalized the previous result to the whole space case.  The first author is very grateful for having had the opportunity to discuss the papers \cite{MR1321370,MR1211782,MR1392991,MR1275397,MR1105532} with Walter Strauss on several occasions while he was a graduate student at Brown University.

\subsection{The literature}
In this section we will describe a selection of closely related additional results about the relativistic Landau equation.  Lemou  \cite{MR1773932} in 2000 studied the linearized relativistic Landau collision operator.  Strain and Guo, in 2004 \cite{MR2100057}, proved the global existence of unique classical solutions to the relativistic Landau-Maxwell system with initial data that is close to the relativistic Maxwellian equilibrium.  Then Hsiao and Yu in 2006 \cite{MR2289548}  proved the existence of global classical solutions to the initial value problem for the simpler relativistic {L}andau equation with nearby relativistic Maxwellian initial data in the whole space.    Yu \cite{MR2514726} in 2009 proved the $C^\infty$ smoothing effects for the relativistic Landau-Maxwell system with nearby equilibrium initial data under the assumption that the electric and magnetic fields are infinitely smooth.  Further for relativistic Landau-Poisson equation the smoothing effects were shown in \cite{MR2514726} without additional assumptions.  In 2010 Yang and Yu in \cite{MR2593052}, proved the Hypocoercivity of the relativistic Landau equations.  Then in 2012, Yang and Yu \cite{MR2921603} proved the existence and uniqueness of global in time classical solutions to the relativistic Landau-Maxwell system in the whole space $\mathbb{R}^3_x$ for initial data which is nearby to the relativistic Maxwellian.

The non-relativistic Landau equation has experienced a much larger amount of mathematical study in comparison.      We will mention only a small sample of results that are closely related to this paper.  In 1977 \cite{MR0470442} proved the existence of a local in time bounded solution.   In 2002 \cite{MR1946444} Guo proved the global existence and uniqueness of classical solutions to the spatially dependent Landau equation with nearby Maxwellian equilibrium initial data.  The large time decay rates were shown in \cite{MR2209761}.    Recent developments in \cite{MR3625186} give an understanding of the case with a mild velocity tail on the initial data.   Further \cite{MR3670754} performs a numerical study on the large time decay rate in terms of the $2/3$ law as in \cite{MR2366140}.   We further reference \cite{MR2904573,MR3101794}.

Now for the spatially homogeneous non-relativistic Landau equation, in \cite{MR1737547,MR1737548} Desvillettes and Villani proved the large data global well-posedness and smoothness of solutions for the Landau equation with hard potentials.  In 1998 in \cite{MR1650006} Villani proved the existence of weak H-solutions of the spatially homogeneous Landau equation with Coulomb potential.  Later in 2015  \cite{MR3369941} Desvillettes proved an Entropy dissipation estimate for the Landau equation, and used that estimate to conclude that the H-solutions are actually true weak solutions.   Further developments can be found in \cite{MR3557719,MR3614751}.  Also \cite{MR3158719} proved $L^p$ estimates for the Landau equation with soft potentials.  In \cite{MR3375485} apriori estimates for the Landau equation with soft potentials including the Coulomb case are proven.    

Recently \cite{MR3582250} proved upper bounds for certain parabolic equations,  including the spatially dependent non-relativistic {L}andau equation after conditionally assuming the local conservation laws to be bounded.  And \cite{HarnackLandau} proves a Harnack inequality for solutions to kinetic Fokker-Planck equations with rough coefficients and applies that to the spatially dependent Landau equation to obtain a $C^\alpha$ estimate, assuming that the local conservation laws are bounded.  In \cite{MR3599518} other estimates are proven for the homogeneous 
{L}andau equation with {C}oulomb potential.

\subsection{Notations}  In this section we will introduce several notations that will be used throughout the rest of the article.  Let $\mathcal{P}(\mathbb{R}^d)$ be the set of probability measures on $\mathbb{R}^d$ ($d\ge 1$) and $\mathcal{P}_r(\mathbb{R}^d)$ ($r \geq 1$) be the subset of $\mathcal{P}(\mathbb{R}^d)$ with finite $r-$th moments, i.e. 
$$
\mathcal{P}_r(\mathbb{R}^d) 
\eqdef 
\Big\{ f \in \mathcal{P}(\mathbb{R}^d) \Big \vert \int_{\mathbb{R}^d} |x|^r f(\ud x) < \infty \Big\}.    
$$
We introduce the Wasserstein  distance  on  $\mathbb{R}^3$  to compare  two weak solutions to the relativistic Landau equation \eqref{RelLandau} as in \eqref{weakform} below.  For two probability measures  $f, g \in \mathcal{P}_r(\mathbb{R}^3)$, their  $r-$Wasserstein distance $\mathcal{W}_r(f, g)$  is defined  as  
\begin{equation*}
\begin{split}
\mathcal{W}_r(f, g) & =\inf_{R \in\ProbJ(f, g)}\bigg(\int_{\mathbb{R}^{3}\times\mathbb{R}^{3}}|x-y|^r  R(dx, dy)\bigg)^{1/r} \\& =\inf_{X\sim f,\, Y\sim g }\bigg(\mathbb{E}[|X-Y|^r]\bigg)^{1/r},
\\
\end{split}
\end{equation*}
where the first infimum is taken over $R \in \ProbJ(f, g)$.  Here $\ProbJ(f, g)$ is the set of joint probability measures on ${\mathbb{R}^{3}}\times{\mathbb{R}^{3}}$ with marginals $f$ and $g$ respectively.  Further $X\sim f$ means that $X$ is an $\mathbb{R}^3$ valued random variable with law $f$, and $Y\sim g$ is similarly defined.  Then the infimum in the second inequality above is over all possible couplings $(X,Y)$ of random variables with $f$ and $g$ as  their marginal  laws respectively.   It is known that $(\mathcal{P}_2,\mathcal{W}_2)$ is a Polish space whose topology is a bit stronger than the weak topology.   It is further known that the infimum above is reached in the sense that for $f, g \in \mathcal{P}_2$, then there exists $R \in \ProbJ(f, g)$ and $X\sim f$, $Y\sim g$ such that  
$$
\mathcal{W}_2(f, g)  
=\bigg(\int_{\mathbb{R}^{3}\times\mathbb{R}^{3}}|x-y|^2  R(dx, dy)\bigg)^{1/2} 
=\bigg(\mathbb{E}[|X-Y|^2]\bigg)^{1/2}.
$$
See \cite{MR2459454} for a thorough introduction of the Wasserstein  distance.

In this article in particular we will use the $2-$Wasserstein  distance, which is  $\mathcal{W}_2(F_t,  \tilde F_t)$, to quantify the distance of two weak solutions \eqref{weakform} $(F_t)_{t \in [0, T]}$ and $( \tilde F_t)_{t \in [0, T]}$ to the relativistic Landau equation \eqref{RelLandau}.  In particular for any $s \in [0, T]$, choose  $R_s \in \ProbJ(F_s, \tilde F_s)$ to be the unique probability measure on $\mathbb{R}^3 \times \mathbb{R}^3$ with marginals $F_s$ and $\tilde F_s$ such that 
\begin{equation}\label{Wass2.def}
\mathcal{W}_2^2(F_s, \tilde F_s) = \int_{\mathbb{R}^3 \times \mathbb{R}^3}  |p- \tilde p|^2 R_s (\ud p, \ud \tilde p).  
\end{equation}
We will use this distance extensively throughout the paper.

We will now define the weighted Lebesgue spaces $L^r_s(\threed)$ (with $r \geq 1$ and $s\in \mathbb{R}$) as follows
$$
\| f\|_{L^r_s(\threed)} \eqdef
 \| \langle \cdot \rangle^s  f\|_{L^r(\mathbb{R}^3)}
=
\left(
\int_{\threed} dp ~  \langle p \rangle^{rs} | f(p)|^r 
\right)^{1/r},
$$
where
$
 \langle p \rangle \eqdef \left(1 + |p|^2 \right)^{1/2}
$
with the corresponding standard definition for $L^\infty_s(\threed)$.  
We use the definition $ L^r_s(\mathbb{R}^3) = \{f: \mathbb{R}^3 \to \mathbb{R}, \, \| f \|_{L^r_s(\mathbb{R}^3)} < +\infty \}$. 
We write $L^r_s$ when there is no risk of confusion about the domain.  Further we will denote $L^r_0 = L^r$ (when $s=0$ with $r \geq 1$) throughout the article.

We will now define the following useful functions:
\begin{definition}\label{DefPsi}Define the  increasing continuous function $\Psi: [0, \infty) \to [0, \infty)$ as 
\[
\Psi (x) = x \left( 1 - \IN_{ \{ 0\le x \leq 1 \} }  \log x \right) . 
\]
Define the  concave increasing continuous function $\Theta: [0, \infty) \to [0, \infty)$ as 
\[
\Theta(x) = \begin{cases}
x(1 - \log x) & \mbox{if}\ x \in [0, \frac{1}{2}]; \\
x \log 2 + \frac{1}{2} & \mbox{if}\ x \geq \frac{1}{2}. 
\end{cases}
\]
\end{definition}

Note that for any $x \geq 0$, $\Psi(x)/2 \leq \Theta(x) \leq 2 \Psi(x)$. Since $\Theta$ is concave, we conclude for any $f\ge 0$ from Jensen's inequality that
\begin{equation}
\label{Jensen}
\int \Psi \circ f  \ud \mu \leq 2 \int \Theta \circ f  \ud \mu \leq  2 \Theta \left( \int f  \ud \mu \right) \leq 4 \Psi \left( \int f \ud \mu \right),  
\end{equation}
where $\mu$ is any probability measure.

For two quantities $A$ and $B$, we will use the notation $A \lesssim B$ to mean that there exists a positive inessential constant $C>0$ such that $A \le C B$.  Then $A \approx B$ means that $A \lesssim B$ and $B \lesssim A$.  We will use the notation $C>0$ and also $c>0$ to denote positive inessential constants whose value may change from line to line.

\subsection{Weak solutions to the relativistic Landau equation \eqref{RelLandau}}

For a test function $\varphi \in C^2_b(\mathbb{R}^3)$, we can formally integrate by parts the relativistic Landau equation \eqref{RelLandau} with kernel \eqref{ColKer} to obtain the following weak formulation
\[
\begin{split}
\frac{\ud }{\ud t} &  \int \varphi(p) F(t, p) \ud p =  \frac{1}{2} \int_{\mathbb{R}^3 \times \mathbb{R}^3}  F(p) F(q) \left(\Phi^{ij}(p, q) \partial_{ij}^2 \varphi(p)  \right) \ud p \ud q \\
 + &  \frac{1}{2} \int_{\mathbb{R}^3 \times \mathbb{R}^3}  F(p) F(q) \partial_{p_i} \varphi(p) \left( \partial_{p_j} \Phi^{ij}(p, q) -  \partial_{q_j} \Phi^{ij}(p, q) \right) \ud p \ud q,
\end{split} 
\]
where we sum over $i$ and $j$. Now we introduce 
\begin{equation}
\label{LDef}
L \varphi (p, q) = \frac{1}{2} \sum_{i, j=1}^3 \Phi^{i j}(p, q) \partial_{ij}^2 \varphi(p)  + \sum_{i=1}^3  B_i(p, q) \partial_{p_i} \varphi(p),
\end{equation}
where 
\begin{equation}
\label{BiDef}
B_i(p, q)  = \frac{1}{2} \sum_{j=1}^3 \left( \partial_{p_j}  - \partial_{q_j} \right) \Phi^{ij}(p, q) =   \Lambda(p, q) (\rho +2) (q_i - p_i), 
\end{equation}
This formulation is analogous to the  weak formulation of the non-relativistic Landau equation given in \cite[Equation (5)]{MR2718931}.  We obtain the expression \eqref{BiDef} by using the following lemma

\begin{lemma} [Strain and  Guo  \cite{MR2100057}] \label{LemmaOfB} One has 
\[
\sum_{j=1}^3 \partial_{p_j} \Phi^{ij}(p, q) = 2 \Lambda(p, q) \left( (\rho +1) q_i - p_i) \right), 
\]
and 
\[
\sum_{j=1}^3 \partial_{q_j} \Phi^{ij}(p, q) = 2 \Lambda(p, q) \left( (\rho +1) p_i - q_i) \right), 
\]
\end{lemma} 

The proof of Lemma \ref{LemmaOfB} can be found in \cite[Lemma 3]{MR2100057}.  The detailed calculations which show the derivation of the weak formulation of the relativistic Landau equation \eqref{RelLandau} are contained in \cite{StrainTas}.  We give the following definition of a weak solution:

\begin{definition}[Weak solutions to the relativistic Landau equation \eqref{RelLandau}]\label{weak.def} 
 We  call  $(F_t)_{t \in [0, T]}$ a weak solution to the relativistic Landau equation \eqref{RelLandau} with initial data $F_0 $ a probability measure,   provided that  
$$
(F_t)_{t \in [0, T]} \in L^\infty([0, T], \mathcal{P}_1) \cap L^1([0, T], L^\infty), 
$$
and 
for any $\varphi \in C_b^2(\mathbb{R}^3)$ and any $t \in [0, T]$, it holds that
\begin{equation}
\label{weakform}
\int_{\mathbb{R}^3} \varphi(p) F_t( p) \ud p = \int_{\mathbb{R}^3} \varphi(p) F_0(p) \ud p 
%\\
+ \int_0^t \int_{\mathbb{R}^3 \times \mathbb{R}^3} F_s( p) F_s( q) L \varphi(p, q) \ud p \ud q \ud s, 
\end{equation}
where $L$ is defined in \eqref{LDef}. 
\end{definition}

In the recent work \cite{StrainTas}, the entropy dissipation estimate was shown for weak solutions to the relativistic Landau equation.  From the Sobolev inequality, then the entropy dissipation estimate implies the gain of 
$\nabla \sqrt{f} \in L^1([0,T]; L^2(\mathbb{R}^3))$  for a weak solution. 
Then, with that estimate, the global existence of a standard weak solution was established.  After that the propagation of any high order polynomial moments was shown in the sense that 
\begin{align}\notag
M_k(f,T) = \esssup_{t\in[0,T]}  \int_{\threed} f(t,p) (1+|p|^2)^k dp \le  C<\infty,
\end{align}
holds as long as it holds initially.  We refer to \cite{StrainTas} for the complete details.

\subsection{Main Results} The majority of the work in this article will go towards establishing the following integral inequality in Proposition \ref{PropIntegral}.  After that we will use this integral inequality to establish the uniqueness in our main theorem.

\begin{proposition}[Main Integral Inequality] \label{PropIntegral} For any two weak solutions $(F_t)_{t \in [0, T]}$ and $(\tilde F_t)_{t \in [0, T]}$ to the relativistic Landau equation \eqref{RelLandau}, as in Definition \ref{weak.def}, there exists  a bounded function $\rho: [0, T]\mapsto [0, \infty)$, such that for any $t \in [0, T]$, 
\[
\mathcal{W}_2^2 (F_t, \tilde F_t) \leq \rho(t), \quad \rho(t) \leq \mathcal{W}_2^2(F_0, \tilde F_0) +  \int_0^t \widetilde C(F_s, \tilde F_s)  \Psi(\rho(s)) \ud s, 
\]
where $\Psi$ is from Definition \ref{DefPsi} and 
\begin{equation}\label{DefCFsTildeFs}
\widetilde C(F_s, \tilde F_s)  \eqdef  c \left(  \|  F_s\|_{L^\infty_7 \cap L^1_7} +  \|  \tilde  F_s\|_{L^\infty_7 \cap L^1_7}  +1\right).
\end{equation}

\end{proposition}

We will use the integral inequality above to prove our main theorem

\begin{theorem}[Main Theorem] \label{MainTheorem} Let $T> 0$. (i) Given the initial data  $F_0$ to the relativistic Landau Eq. \eqref{RelLandau} satisfying  $F_0(p) \in L^\infty_7 \cap L^1_7$, then there exists at most one weak solution to \eqref{RelLandau} starting from $F_0(p)$  obeying the following moment bounds 
\[
\int_0^T \| F(s)\|_{L^\infty_7 \cap L^1_7}  \ud s < \infty. 
\]   
(ii) Suppose that $(F_t)_{t \in [0, T]}$ and $(F_t^n)_{t \in [0, T]}$  ($n\ge 1$) are weak solutions to Eq. \eqref{RelLandau}, satisfying 
\[
\sup_{n\ge 1} \int_0^T  \left( \|  F(s)\|_{L^\infty_7 \cap L^1_7}  +  \| F^n(s)\|_{L^\infty_7 \cap L^1_7}  \right) \ud s < \infty.
\] 
 If initially $\lim_n \mathcal{W}_2(F_0, F_0^n) =0$, then $\lim_n \sup_{t \in [0, T]} \mathcal{W}_2(F_t, F_t^n) =0$. 
\end{theorem}

We point out that Theorem \ref{MainTheorem} in particular applies to the case when  $F_t$ and   $\tilde F_t$ are  the steady states such as
$$
J(p) = \frac{1}{4\pi} e^{-\pZ}
$$
These are called relativistic Maxwellians or the J{\"u}ttner distributions.

\subsection{A summary of the uniqueness argument} 
Theorem \ref{MainTheorem}  follows the scheme introduced by Fournier and Gu\'erin in \cite{MR1970275, MR2398952,MR2502525, MR2718931}, which is based on the stochastic representation of the regular homogeneous Landau equation.  See also  the probabilistic interpretation of the Boltzmann equation in \cite{MR512334,MR791288}.   Our result can be seen as the special relativistic counterpart of the result in  Fournier \cite{MR2718931}. 

However  the relativistic case is algebraically  much more challenging. To the best of the authors' knowledge,   the  stochastic representation  \eqref{SDE1} and \eqref{SDE2} with the particular coefficient matrix $\Sigma$ in Proposition \ref{CoeffBM}   is new in  the literature.   We now give an overview of the ideas in the proof of Theorem \ref{MainTheorem}.  

\noindent$\bullet$ We use the $2-$Wasserstein distance to measure the distance between two weak solutions to \eqref{RelLandau}, since trivially  
\[
\mathcal{W}_2(F_t, \tilde F_t) \leq \mathbb{E}(|P_t - \tilde P_t|^2). 
\]
where $F_t = \mbox{Law}(P_t)$ and $\tilde F_t = \mbox{Law}(\tilde P_t)$. 
Thus we can instead control an easier quantity $\mathbb{E}(|P_t - \tilde{P}_t|^2)$,  whose evolution   is simply  given by  It\^o's formula.   We refer to the proof of Proposition \ref{PropIntegral} for example below.

\noindent$\bullet$  The  drift term  in the evolution of $\mathbb{E}(|P_t - \tilde P_t|^2)$ can be controlled by  Proposition \ref{EstDoubleB}, while the  diffusion term  can be  controlled by  Proposition \ref{PropEstIntegralSingle}. Those integral estimates are established in Section \ref{sec:EstInt}, which further deeply depend on the crucial point-wise estimates for $\Phi$, $\Sigma$ and $B$ established in Section \ref{sec:PointEst}. Section \ref{sec:PointEst} and Section \ref{sec:EstInt} are the most technical parts in  this article. 
 
\noindent$\bullet$ In order to conclude the estimate for $\mathcal{W}_2(F_t, \tilde F_t)$, we choose very particular initial random variables  (see \eqref{InitalLaws}) and white noises $W= W( \ud p, \ud \tilde p, \ud s)$ with covariance measure $R_s(\ud p, \ud \tilde p) \ud s$ where the meaure $R_s(\ud p, \ud \tilde p) \ud s$ has marginals $F_s$ and $\tilde F_s$.

\noindent$\bullet$  To prove the crucial integral inequality (Proposition \ref{PropIntegral}),  we need  the uniqueness of the coupled SDEs \eqref{SDE1} and \eqref{SDE2}, i.e. Proposition \ref{PropUniquenessSDEs}.  Except for the standard techniques in SDEs, we also  apply Theorem B.1 in Horowitz-Karandikar \cite{MR1042343} (see also Theorem 5.2 in Bath-Karandikar \cite{BK93})  to obtain the uniqueness of  the linear relativistic Landau  \eqref{LinearLandau}, which is possible only due to our pointwise estimates of the coefficients such as
\[
|\Phi(p, q) | \leq c\left( 1 + q^0 |p-q|^{-1} \right), \quad |B(p, q)| \leq c \left( 1+ q^0 |p-q|^{-2}   \right).
\]
These estimates are proven in Lemma \ref{lem:EstOfPhi} and Lemma \ref{lem:BTrivial} below.

\section{Stochastic representation}
In this section we will present a stochastic differential equation that we will use to represent the relativistic Landau equation \eqref{RelLandau}.  As a first step, we shall further decompose the relativistic Landau kernel \eqref{ColKer} as follows. 

\subsection{Decomposition of  the kernel $\Phi$} Now we will present a decomposition of the relativistic Landau kernel \eqref{ColKer}.  The crucial point of this section is to introduce a new matrix decomposition of the kernel in \eqref{sigmaM.def}.  This matrix decomposition will allow us to present a useful stochastic representation of weak solutions to the relativistic Landau equation.  

\begin{proposition} 
 \label{SpecOfPhi}  \cite{StrainTas}.
The relativistic Landau kernel $\Phi = (\Phi^{ij})$ from \eqref{ColKer} is symmetric, positive semi-definite with null space spanned by $(\frac{p}{p^0} - \frac{q}{q^0})$.  The matrix $S$ can be decomposed as the difference of two orthogonal projectors, {\em i.e.}
\[
S= \Pi_1 - \Pi_2, 
\]
where 
\[
\Pi_1 = |q^0 p - p^0 q|^2 Id - (q^0 p - p^0 q) \otimes (q^0 p - p^0 q); 
\]
and 
\begin{equation*}
\Pi_2= |p \times q|^2 Id - |q|^2 p \otimes p - |p|^2 q \otimes q + (p \cdot q) \ ( p \otimes q + q \otimes p) 
%\\
= (p \times q) \otimes (p \times q). 
\end{equation*}
\end{proposition}

Above $Id$ is the standard $3 \times 3$ identity matrix.  The above proposition is proven in \cite{StrainTas}.  We give a different elementary proof here for the sake of completeness.
 
\begin{proof}  Recall from \eqref{Smatrix.def} that $$S = \rho \tau Id - (p -q) \otimes (p-q) + \rho (p \otimes q + q \otimes p)$$ and  that $$  \rho \tau = |q^0 p - p^0 q|^2 - |p \times q|^2. $$
The proof will then  be done by direct calculation as seen in \cite{StrainTas}.   But we remark that both $\Pi_1$ and $\Pi_2$ are orthogonal projectors. 
Indeed,
\[
\frac{1}{|v_1|^2}\ \Pi_1 = P_{\bot v_1} = Id - \frac{v_1}{|v_1|} \otimes \frac{v_1}{|v_1|}
\]
and 
\[
\frac{1}{|v_2|^2}\  \Pi_2 = \frac{v_2}{|v_2|} \otimes \frac{v_2}{|v_2|}. 
\]
where $v_1 = q^0 p - p^0 q$ and $v_2 = p \times q$.  The last equality for $\Pi_2$ is guaranteed by the observation $ \Pi_2\  p = \Pi_2\  q =0$ and $\Pi_2 \ v_2 = |v_2|^2 v_2$. Indeed, 
\[
\begin{split}
\Pi_2 \ p  & = |p \times q|^2 p - |p|^2 |q|^2 p - |p|^2 (p \cdot q) q + (p \cdot q)^2 p + (p \cdot q) |p|^2 q   \\
& = \left( | p \times q|^2 + (p \cdot q)^2 - |p|^2 |q|^2 \right) p  = 0. 
\end{split}
\]
Similarly, one can show that $\Pi_2 \  q=0$ and $\Pi_2 v_2 = |v_2|^2 v_2$. 
\end{proof}

\medskip
Note that the rank $2$ projector   $\Pi_1$ above has the same structure as the non-relativistic Landau kernel $a(z) = \frac{1}{|z|^3} (|z|^2 Id - z \otimes z)$.  With this observation we then set $ \Pi_1 = \sigma_{\Pi_1} \,  \sigma_{\Pi_1}^\top$ with 
\begin{equation} \label{SigmaPi1}
\sigma_{\Pi_1} = \begin{bmatrix}
q^0 p_2 -p^0 q_2 & -(q^0 p_3 -p^0 q_3) & 0 \\
-(q^0 p_1 - p^0 q_1)  & 0 & q^0 p_3 -p^0 q_3 \\
0 & q^0 p_1 -p^0 q_1 & -(q^0 p_2 -p^0 q_2)
\end{bmatrix}, 
\end{equation}
which is the analog of  $\sigma(z)$, a square root of $a(z)$, as in \cite[Equation (6)]{MR2718931}.

The other  rank $1$ projector  $\Pi_2 = (p \times q) \otimes (p \times q)$  in general can be  written as  $\Pi_2 = \sigma_{\Pi_2} \ \sigma_{\Pi_2}^\top$ with $\sigma_{\Pi_2} = (p \times q) \otimes u$  for any unit vector $u \in \mathbb{S}^2$. But the particular choice of $u= \frac{1}{|p|} (p_3 \  p_2\  p_1)^\top$  will be compatible with $\Pi_1$  (or $\sigma_{\Pi_1}$):
\begin{equation}
\label{SigmaPi2}
\sigma_{\Pi_2} \eqdef \frac{1}{|p|} (p \times q) \otimes 
\begin{bmatrix}
p_3 \\p_2 \\p_1
\end{bmatrix}. 
\end{equation}
Here for the sake of clarity we point out that we are using the notation
\begin{equation}
\notag
(p \times q) \otimes \begin{bmatrix}
p_3 \\p_2 \\ p_1 
\end{bmatrix} =
 \begin{bmatrix}
p_3(p_2 q_3 -p_3 q_2) & p_2 (p_2 q_3 -p_3 q_2) & p_1 (p_2 q_3 -p_3 q_2) \\
p_3(p_3 q_1 -p_1 q_3) & p_2 (p_3 q_1 -p_1 q_3) & p_1 (p_3 q_1 -p_1 q_3) \\
p_3(p_1 q_2 -p_2 q_1) & p_2  (p_1 q_2 -p_2 q_1)  & p_1 (p_1 q_2 -p_2 q_1) \\
\end{bmatrix}. 
\end{equation}
We will observe the compatibility of $\sigma_{\Pi_1}$ and $\sigma_{\Pi_2}$ in the following proposition.

\begin{proposition} \label{PropSDecom}The matrix $S = (S^{ij})$ can be written as the following 
\[
S = \sigma_S \ \sigma_S^\top,
\]
where 
\begin{equation}\label{matrixS.eq}
\begin{split}
 \sigma_S =   &~\sigma_{\Pi_1} - \frac{|p|}{p^0 +1} \ \sigma_{\Pi_2}  \\
 = & \begin{bmatrix}
q^0 p_2 -p^0 q_2 & -(q^0 p_3 -p^0 q_3) & 0 \\
-(q^0 p_1 - p^0 q_1)  & 0 & q^0 p_3 -p^0 q_3 \\
0 & q^0 p_1 -p^0 q_1 & -(q^0 p_2 -p^0 q_2)
\end{bmatrix} \\
& \quad \quad - \frac{1}{p^0 +1} (p \times q) \otimes \begin{bmatrix}
p_3 \\p_2 \\p_1
\end{bmatrix}. 
\end{split}
\end{equation}
\end{proposition}

\begin{proof} It is straight forward to check that 
\begin{equation}\label{ProSDeo}
\sigma_{\Pi_1} \begin{bmatrix}
 p_3  \\ p_2 \\p_1
\end{bmatrix} = p^0 (p \times q), \quad \sigma_{\Pi_1} \begin{bmatrix}
q_3 \\q_2 \\q_1  
\end{bmatrix}= q^0 (p \times q).
\end{equation}
Here we define  $\sigma_{\Pi_2}$ by \eqref{SigmaPi2}. Assume that $\sigma_S = \sigma_{\Pi_1} - \mu \sigma_{\Pi_2}$, then  
\[
\begin{split}
\sigma_S \ \sigma_S^\top &  = (\sigma_{\Pi_1} - \mu \sigma_{\Pi_2}) (\sigma_{\Pi_1} - \mu \sigma_{\Pi_2})^\top \\
& = \sigma_{\Pi_1} \sigma_{\Pi_1}^\top - \mu \sigma_{\Pi_1} \sigma_{\Pi_2}^\top - \mu \sigma_{\Pi_2} \sigma_{\Pi_1}^\top + \mu^2 \sigma_{\Pi_2} \sigma_{\Pi_2}^\top
= S + \Delta,
\end{split}
\]
where 
\[
\begin{split}
\Delta  & = (\mu^2 +1) \Pi_2 - \frac{\mu}{|p|} \sigma_{\Pi_1} \begin{bmatrix}
p_3 \\ p_2 \\p_1
\end{bmatrix} (p \times q)^\top - \frac{\mu}{|p|} (p \times q)  [p_3 \ p_2\  p_1] \sigma_{\Pi_1}^\top \\
& = \left( \mu^2 -  \frac{ 2 p^0}{|p|} \mu + 1 \right) \Pi_2,
\end{split}
\]
where the last equality is ensured by \eqref{ProSDeo}. The difference $\Delta$ vanishes at $\mu = \frac{p^0 \pm 1}{|p|}$.  In particular, we can choose $\mu = \frac{p^0 -1}{|p|}= \frac{|p|}{p^0 +1}$ such that  $S= \sigma_S \ \sigma_S^T.$
\end{proof}

We finally obtain a useful formula for the  square root $\Sigma(p, q)$ of the relativistic Landau kernel $\Phi(p, q)$

\begin{proposition} \label{CoeffBM} The relativistic Landau kernel matrix $\Phi(p, q)$ can be decomposed as 
\[
\Phi = \Sigma\  \Sigma^\top,
\]
where 
\begin{multline}\label{sigmaM.def}
 \Sigma = \frac{\rho+1}{(p^0 q^0)^{1/2}} (\rho \tau)^{-3/4} \ \sigma_S  
\\
=   \frac{\rho+1}{(p^0 q^0)^{1/2}} (\rho \tau)^{-3/4} 
\Bigg\{ 
\begin{bmatrix}
q^0 p_2 -p^0 q_2 & -(q^0 p_3 -p^0 q_3) & 0 \\
-(q^0 p_1 - p^0 q_1)  & 0 & q^0 p_3 -p^0 q_3 \\
0 & q^0 p_1 -p^0 q_1 & -(q^0 p_2 -p^0 q_2)
\end{bmatrix} 
\\
 \quad \quad - \frac{1}{p^0 +1} (p \times q) \otimes \begin{bmatrix}
p_3 \\p_2 \\p_1
\end{bmatrix} \Bigg\}. 
\end{multline}
\end{proposition}

\begin{remark} The matrix in the last line can be changed to 
\[
\frac{1}{q^0 +1} (p \times q) \otimes \begin{bmatrix}
q_3 \\q_2 \\q_1
\end{bmatrix},
\]
with $\sigma_{\Pi_2}$ and $\sigma_{S}$ changed correspondingly. 
\end{remark}

\begin{remark} 
Our choice of $\Sigma$, such that  $\Sigma \Sigma^\top = \Phi$, is of course not unique. But our choice in Proposition \ref{CoeffBM} appears to be more compatible with our uniqueness argument given later on in this article. We further expect this formulation can be useful in other scenarios in the future. 
\end{remark}

\subsection{Stochastic representation and weak solutions}

In this section we present the stochastic representation of weak solutions to the relativistic Landau equation  \eqref{RelLandau}.   This is only a brief summary and setup, in general for the full details of SDEs we refer to the detailed discussions in \cite{MR876085}.  

We introduce two coupled Landau stochastic processes, say $(P_t)_{t \in [0, T]}$ and $(\tilde P_t)_{t \in [0, T]}$, whose laws are weak solutions $(F_t)_{t \in [0, T]}$ and $(\tilde F_t)_{t \in [0, T]}$  to the relativistic Landau Eq. \eqref{RelLandau} respectively.    For any $s \in [0, T]$, choose  $R_s \in \ProbJ(F_s, \tilde F_s)$ to be the (unique) probability measure on $\mathbb{R}^3 \times \mathbb{R}^3$ with marginals $F_s$ and $\tilde F_s$ such that \eqref{Wass2.def} holds. Indeed, $R_s$ is the optimal transport plan which gives the $2-$Wasserstein  distance for  $F_s$ and $\tilde F_s$ in \eqref{Wass2.def}.  Consider a 3D white noise $W( \ud p, \ud \tilde p, \ud s)$ on $\R^3 \times \R^3 \times [0, T]$ with covariance measure $R_s(\ud p, \ud \tilde p) \ud s$. 

Then choose two $\R^3-$valued random variables $P_0$ and $\tilde P_0$ with laws $F_0$ and $\tilde F_0$  respectively,   independent of the white noise $W$, such that initially
\begin{equation}\label{InitalLaws}
\mathcal{W}_2^2 (F_0, \tilde F_0) = \mathbb{E}[|P_0- \tilde P_0|^2]. 
\end{equation}
Then the coupled  $\mathbb{R}^3-$valued  stochastic differential equations (SDEs) are 
\begin{equation}
\label{SDE1}
P_t = P_0 + \int_0^t \int_{\R^3 \times \R^3} \Sigma(P_s, p) \,  W(\ud p, \ud \tilde p, \ud s) 
%\\
+ \int_0^t \int_{\R^3} B(P_s, p) F_s (p) \ud p \ud s, 
\end{equation}
and 
\begin{equation}
\label{SDE2}
\tilde P_t = \tilde P_0 + \int_0^t \int_{\R^3 \times \R^3} \Sigma(\tilde P_s, \tilde p) \,  W(\ud p, \ud \tilde p, \ud s) 
%\\
+ \int_0^t \int_{\R^3} B( \tilde P_s, \tilde p) \tilde F_s ( \tilde p) \ud  \tilde p \ud s,  
\end{equation}
where $\Sigma$ is defined in \eqref{sigmaM.def} and $B$ is defined in \eqref{BiDef} respectively. Note that the filtration is  $\mathcal{F}_t= \sigma\{ P_0, \tilde P_0, W(A \times [0, s]), s \in [0, t], A \in \mathcal{B}(\mathbb{R}^3 \times \mathbb{R}^3)\}$.   Here $\mathcal{B}(\mathbb{R}^3 \times \mathbb{R}^3)$ is the Borel sigma algebra.

Given a weak solution $(F_t)_{t \in [0, T]}$ to \eqref{RelLandau}, then \eqref{SDE1} can be regarded as a classical stochastic differential equation (SDE). Indeed, \eqref{SDE1} can be rewritten as 
\begin{equation}
\label{LandauProbabilistic}
P_t = P_0 + \int_0^t \Sigma_{F_s} (P_s) \ud B_s + \int_0^t B_{F_s} (P_s) \ud s, 
\end{equation}
where $(B_t)_{t \in [0, T]}$ is a standard 3D Brownian motion and  
\[
B_{F_s} (p) = \int_{\R^3} B(p, q) F_s(q) \ud q, 
\]
and $\Sigma_{F_s}(p) $ is a square root of $\int_{\mathbb{R}^3} \Phi(p, q) F_s(q) \ud q$. Eq. \eqref{SDE1} is nothing but the standard probabilistic interpretation of the relativistic Landau \eqref{RelLandau}.   The same argument applies to Eq. \eqref{SDE2}. 

The white noise $W(\ud p, \ud \tilde p, \ud s)$ allows us to couple two Brownian motions (one in Eq. \eqref{SDE1} or its counterpart in Eq. \eqref{LandauProbabilistic},  and the other  in Eq. \eqref{SDE2}) such that the two solutions $(P_t)_{t \in [0, T]}$ and $(\tilde P_t)_{t \in [0, T]}$ (or their laws $(F_t)_{t \in [0, T]}$ and $( \tilde F_t)_{t \in [0, T]}$ respectively) would remain close to each other.

  \subsection{Outline of the rest of the article} The rest of the article is organized as follows.  In Section \ref{sec:PointEst} we prove the useful pointwise estimates of $\Sigma$ and $B$.  Then in Section \ref{sec:EstInt} we prove crucial estimates of the integrals of the quantities $\Sigma$ and $B$.    After that in Section \ref{gronwall.section} we explain a known useful generalized Gronwall inequality.  Next in Section \ref{integral.ineq.section} we give the proof of the crucial Proposition \ref{PropIntegral}.    Then finally in Section \ref{proof.main.sec} we finally prove our Main Theorem \ref{MainTheorem}.

\section{Estimates of the coefficients $\Sigma$ and $B$}\label{sec:PointEst}

In this section we will give new pointwise estimates for the important quantities in the relativistic Landau equation, $\Sigma$ and $B$, including estimates of their differences.  We initially state a useful inequality for \eqref{rho.def.eq} which is taken from \cite{MR1211782}:
\begin{proposition}\label{RhoEstProp}
Let $p, q \in \mathbb{R}^3$ and define  $\rho = p^0 q^0 - p \cdot q -1$  as in \eqref{rho.def.eq}.  Then  
$$
0 \leq \rho <  2 p^0 q^0
$$ 
and 
\begin{equation}
\label{EstimateOfRho}
\frac{|p-q|^2 + |p \times q|^2}{ 2 p^0 q^0} \leq \rho(p, q) \leq \frac{1}{2} |p-q|^2. 
\end{equation}
\end{proposition}

\begin{proof}
The lower bound for $\rho$  is a direct consequence of the identity \eqref{rho.def.eq} and the fact that $p \cdot q + 1 \leq p^0 q^0$. Direct computations give  that
\begin{equation}
\notag
\rho= \frac{|p-q|^2 - |p^0 - q^0|^2}{2},
\end{equation}
which then implies the upper bound for $\rho$. 
\end{proof}

In the following definition we introduce a splitting of $\mathbb{R}^3 \times \mathbb{R}^3$ which will be crucial in the remainder of this article.

\begin{definition}[Splitting of double space $\mathbb{R}^3 \times \mathbb{R}^3$] \label{DefSetA} We define the subset $A$ of $\mathbb{R}^3 \times \mathbb{R}^3$ as 
\[
A  \eqdef  \{ (p, q) \vert (p^0 q^0)^{1/2} \geq |p-q|\}.  
\]
Furthermore, we define the indicator function 
\[
\IN_A(p, q) = 1- \IN_{A^c} (p, q) = \begin{cases}
1 & \mbox{if} \ (p, q) \in A, \\
0 & \mbox{if} \  (p, q) \notin A. 
\end{cases}
\]
\end{definition}

In the following remark we will explain pointwise estimates that follow from these splittings.

\begin{remark}\label{pointwise.remark}
If $(p, q) \in A$, then $p^0 q^0 \geq (p^0 q^0)^{1/2} \geq |p-q|$ since $p^0, q^0 \geq 1$.  Furthermore, if $(p, q) \in A$, then 
\[
p^0 \leq |p^0 -q^0| + q^0 \leq |p-q| + q^0 \leq (p^0 q^0)^{1/2} + q^0 \leq \frac{p^0 + q^0}{2} + q^0
 \]
 which leads to $p^0 \leq 3 q^0$. By symmetry of $p$ and $q$, one can conclude that 
 \begin{equation}\label{AProperty}
 \frac{1}{3 } q^0 \leq p^0 \leq 3 q^0, \quad \mbox{if}\  (p, q) \in A. 
 \end{equation}
Alternatively if $(p, q) \notin A$, then $|p -q | \geq (p^0 q^0)^{1/2} \geq 1$.
\end{remark}

The set $A$ and its complement $A^c$ will be used many times in this article.   Now we will use the estimate (\ref{EstimateOfRho}) for $\rho$ and recall that $\tau = \rho +2$.  We will use those observations in the proof of the following lemma.

\begin{lemma}[Estimate of $\Lambda$] \label{EstimateOfLamada} The coefficient $\Lambda(p, q) = \frac{(\rho +1)^2}{p^0 q^0} (\rho \tau)^{-3/2}$ can be estimated as 
\[
0 \leq \Lambda(p, q)  \leq c  \left(  \left(p^0 q^0 \right)^{1/2} |p-q|^{-3} \IN_A(p, q) +   |p-q|^{-2} \IN_{A^c}(p, q) \right).
\]
\end{lemma}

\begin{proof}
If $ 0 \leq \rho(p, q)   \leq 1$, then from \eqref{lambda.def} and \eqref{rho.def.eq} we have
\begin{equation*} 
\Lambda(p, q) \leq \frac{4}{p^0 q^0} \frac{(p^0 q^0)^{3/2}}{|p-q|^3} = 4 (p^0 q^0)^{1/2}\  |p-q|^{-3},
\end{equation*}
since $
\rho \geq \frac{|p-q|^2 }{ 2 p^0 q^0}
$
by Proposition \ref{RhoEstProp}, $\tau = \rho + 2 \geq 2$ and $\rho +1 \leq 2$.

 Otherwise if  $\rho \geq 1$, then 
$\rho +1 \leq   2 \rho$ and similarly 
\[
\Lambda(p, q) \leq \frac{1}{p^0 q^0} \frac{(2 \rho)^2}{ (\rho^2)^{3/2}} = \frac{4}{p^0 q^0 \rho } \leq \frac{8}{|p-q|^2}. 
\]  
Then 
\[
\Lambda(p, q) \lesssim \max\{ (p^0q^0)^{1/2} |p-q|^{-3}, |p-q|^{-2}  \}. 
\]
This completes the proof after using Definition \ref{DefSetA} and Remark \ref{pointwise.remark}. 
\end{proof}

We remark that Lemma \ref{EstimateOfLamada} shall be compared to  \cite[Lemma 12]{StrainTas}. For instance, if  $\rho \geq 1$, then  using 
$\rho +1 \leq   2 \rho$, $\tau \ge 2$ again, one obtains 
\[
\Lambda(p, q) \leq \frac{1}{p^0 q^0} \frac{(2 \rho)^2}{ \rho^{3/2}} = \frac{4\rho^{1/2} }{p^0 q^0 } \lesssim \frac{(p^0 q^0)^{1/2} }{p^0 q^0 }
\lesssim \frac{4 }{(p^0 q^0)^{1/2} }. 
\] 
However, the estimates in Lemma \ref{EstimateOfLamada} would be compatible with the indicator $\IN_A$, which will be used throughout this article. 

Now we give a useful pointwise estiamte of the kernel $\Phi$ in \eqref{ColKer}.

\begin{lemma}[Estimate of $\Phi$]\label{lem:EstOfPhi}  The relativistic Landau kernel $\Phi= (\Phi^{ij})$ can be estimated  as 
\[
|\Phi(p, q)| \leq c \left( 1 + (\min\{p^0, q^0 \}) |p-q|^{-1}  \right).  
\]

\end{lemma}

\begin{proof} We recall from \eqref{lambda.def} and \eqref{Smatrix.def} that 
\[
\begin{split}
\Phi^{ij} (p, q) & =   \frac{1}{p^0 q^0 }\frac{(\rho + 1)^2 }{(\rho (\rho +2 ))^{1/2}} \delta_{ij} - \Lambda(p, q) (p_i - q_i) (p_j - q_j) \\
& \, \, + \frac{1}{p^0 q^0} \frac{(\rho +1)^2 }{ \rho^{1/2} (\rho + 2)^{3/2}} (p_i q_j + p_j q_i) \\
& = I_1 + I_2 + I_3. 
\end{split}
\]
First $I_1$ can be estimated similarly as in Lemma \ref{EstimateOfLamada}. When $0 \leq \rho \leq 1$, using $\rho \geq \frac{|p-q|^2}{ 2 p^0 q^0}$,  we have that
\[
\frac{1}{p^0 q^0 }\frac{(\rho + 1)^2 }{(\rho (\rho +2 ))^{1/2}} \leq  c  \frac{1}{p^0 q^0} \frac{(p^0 q^0)^{1/2}}{|p-q|} \leq c (p^0 q^0)^{-1/2} |p-q|^{-1} \leq c |p-q|^{-1}. 
\]
When $\rho \geq 1$, using Proposition \ref{RhoEstProp} and $\rho \le 2 p^0 q^0$, one has 
\[
\frac{1}{p^0 q^0 }\frac{(\rho + 1)^2 }{(\rho (\rho +2 ))^{1/2}} \lesssim \rho \frac{1}{p^0 q^0} \lesssim 1. 
\]
Hence $|I_1| \leq c (1 + |p-q|^{-1})$.  

Using Lemma \ref{EstimateOfLamada} and noting that $p^0 \leq 3 q^0$ for $(p, q) \in A$,  we have that
\[
|I_2| \leq c \left( q^0 |p-q|^{-1} \IN_A(p, q) + \IN_{A^c} (p, q) \right). 
\]
The third term can be estimated as 
\begin{equation}\notag
|I_3| \leq  2 \frac{(\rho+ 1)^2}{ \rho^{1/2} (\rho+2)^{3/2}} 
%\\
\leq c \max\{1, \frac{(p^0 q^0)^{1/2}}{|p-q|} \} \leq c \left( q^0 |p- q|^{-1} \IN_A(p, q) + \IN_{A^c}(p, q) \right). 
\end{equation}
We therefore conclude that
\[
|\Phi(p, q) | \leq c \left( 1 +  q^0 |p-q|^{-1} \right). 
\]
We use the symmetry of $p$ and $q$ to complete this lemma. 
\end{proof}

Now we proceed to  estimate the drift term $B= (B_i)_{i=1}^3$, defined in \eqref{BiDef}. 

\begin{lemma}\label{lem:BTrivial} The drift term $B$ has the pointwise bound 
$$
|B(p, q)|   \leq c \left( \min(p^0, q^0) |p-q|^{-2} \IN_A(p, q)  + \IN_{A^c}(p,q) \right).
$$
\end{lemma}

\begin{proof} One can estimate $B$ as in \eqref{BiDef} with \eqref{lambda.def} as in the following
\[
 |B(p, q)| \leq  |p- q| \frac{1}{ p^ 0 q^0 } \frac{(\rho + 1)^2}{ \rho^{3/2} \tau^{1/2}} . 
\]
If  $0 \leq \rho \leq 1$, using  the estimate $\rho  \geq  \frac{|p-q|^2}{2 p^0 q^0}$ again,  one has 
\[
\begin{split}
& |B(p, q)|  \leq c|p-q| \ \frac{1}{p^0 q^0} \ \rho^{-3/2} \leq c (p^0 q^0)^{1/2} |p-q|^{-2}. 
\end{split}
\]
Otherwise,  if   $\rho \geq 1$, then 
$ \rho \leq \rho + 1 \leq 2 \rho$ and $\rho \leq \rho+2 = \tau \leq 3 \rho$. 
This yields
\[
 |B(p, q)|    \leq c |p-q| \frac{1}{p^0 q^0}. 
\]
To complete the proof we use Definition \ref{DefSetA} and Remark \ref{pointwise.remark}. 
\end{proof}

In the next proposition, we will  estimate for $|B(p, q) - B(\tilde{p}, \tilde{q})|$. 

\begin{proposition}[Pointwise estimate for $B$] \label{PropEstB} For $p, q, \tilde p, \tilde q \in \R^3$, we have
\begin{multline}\notag
 |B(p, q) - B(\tilde p, \tilde q)| \leq c \min   \Big\{  \varphi_B^1(p, q) + \varphi_B^1(\tilde p, \tilde q),  
 \\
 |p - \tilde p | (\varphi_B^2(p, q) +  \varphi_B^2(\tilde p , q) ) +|q- \tilde q| ( \varphi_B^2(q, \tilde p) + \varphi_B^2( \tilde q, \tilde  p) ) \Big\}, 
\end{multline}
where 
\begin{equation}\label{TriB}
\varphi_B^1(p, q)  \eqdef  \min(p^0, q^0) |p-q|^{-2} \IN_A(p, q) + \IN_{A^c}(p, q), 
\end{equation}
and 
\begin{equation}
\label{LipsB} 
\varphi_B^2(p,q)  \eqdef   (q^0)^3 |p-q|^{-3} \IN_A(p, q) + \IN_{A^c}(p, q). 
\end{equation}

\end{proposition}

\begin{proof} By the triangle inequality we have
\[
|B(p, q) - B(\tilde p, \tilde q)| \leq |B(p, q)| + |B(\tilde p, \tilde q)| \leq c \left( \varphi_B^1(p, q) + \varphi_B^1 (\tilde p, \tilde q) \right).
\]
The last inequality is a direct consequence of Lemma \ref{lem:BTrivial}.    In the rest of the proof, we only need to prove the 2nd part, i.e. the  Lipschitz estimate. 

From \eqref{BiDef} and \eqref{lambda.def} we have $B(p, q) = (q - p) \bar \Lambda(p, q)$, where we define the  local notation $\bar 
\Lambda$ as 
\[
\bar \Lambda(p, q) =  (\rho +2 ) \Lambda(p, q) = \frac{1}{p^0 q^0} \frac{(\rho +1)^2}{ \rho^{3/2} (\rho +2)^{1/2}}. 
\]
Similar to the estimates for $\Lambda$ and $B$, we  can also estimate $\bar \Lambda(p, q)$ as 
\begin{equation}\label{EstOfBarLam}
0 \leq \bar \Lambda(p, q) \leq c \left( (p^0 q^0)^{1/2} |p-q|^{-3} \IN_A(p, q) +  (p^0 q^0)^{-1} \IN_{A^c}(p, q)  \right). 
\end{equation}
Indeed, for the case $ \rho \leq 1$, using \eqref{EstimateOfRho}  we have the estimate 
\[
0 \leq \bar \Lambda(p, q) \leq  c \frac{1}{p^0 q^0} \, \rho^{-3/2} \leq c (p^0 q^0)^{1/2} \, |p-q|^{-3}.
\]
When $\rho \geq 1$, similarly one has 
\[
\bar \Lambda(p, q) \leq c \frac{1}{p^0 q^0}, 
\]
since  in this case $\rho, \rho +1 $ and $\rho +2$ are all comparable.

\smallskip
Then we proceed as follows 
\[
\begin{split}
|B(p, q) - B(\tilde{p}, \tilde q)| &\leq  |B(p, q) - B(\tilde{p},  q)| + |B( \tilde p, q) - B(\tilde{p}, \tilde q)| \\
& = I_1 + I_2.
\end{split}
\]
Now we bound $I_1$ and $I_2$ respectively. 

First, one has 
\[
\begin{split}
I_1 & = |(q-p) \bar \Lambda(p, q) - (q-\tilde p) \bar \Lambda(\tilde p, q)|  \\
& \leq  |q-p| |\bar \Lambda(p, q) - \bar \Lambda( \tilde p, q) | + |p - \tilde p | \bar \Lambda(\tilde p, q).
\end{split}
\]
By symmetry one also has 
\[
I_1 \leq |q - \tilde{p}|  |\bar \Lambda(p, q) - \bar \Lambda( \tilde p, q) | + |p-\tilde{p}| \bar \Lambda(p, q). 
\]
Therefore,
\begin{equation}\label{BI_1One}
I_1 \leq \min\left\{ |p-q|, |\tilde{p}-q |\right\}\,  |\bar \Lambda(p, q) - \bar \Lambda( \tilde p, q) | 
%\\
+ |p - \tilde p| \left( \bar \Lambda(p, q) + \bar \Lambda(\tilde p, q) \right). 
\end{equation}
Note that the last term above satisfies a better estimate than \eqref{LipsB}.

Now it suffices to bound $|\bar \Lambda(p, q) - \bar \Lambda(\tilde p, q)|$. To this end we shall first compute the derivative of $ \partial_{p_j}\bar \Lambda$.  Note that 
\[
\partial_{p_j} \rho = \frac{q^0}{p^0} p_j - q_j, \quad \partial_{p_j} \left( \frac{1}{p^0 q^0} \right) = - \frac{1}{(p^0 q^0)^2} \ \frac{q^0}{p^0} p_j.
\]
Define a function $\varphi: \mathbb{R}_{+} \to \mathbb{R}_+$ as 
\[
\varphi(x) = \frac{(x+1)^2}{x^{3/2} (x+2)^{1/2}},
\]
whose derivative is given by
\[
\varphi'(x) =  - \frac{(x+1)(x+3)}{x^{5/2} (x+2)^{3/2}} <0, \quad \mbox{for }\ x >0. 
\]
Therefore, we have that
\[
\partial_{p_j} \bar \Lambda(p, q) = \partial_{p_j} \left( \frac{1}{p^0 q^0}\right) \varphi(\rho ) + \frac{1}{p^0 q^0} \varphi'(\rho ) \partial_{p_j} \rho,
\]
which can be simplified as 
\[
\partial_{p_j} \bar \Lambda(p, q)  = -\bar \Lambda(p, q) \left\{  \frac{p_j}{(p^0)^2} + \frac{\rho +3}{ \rho (\rho +1) (\rho +2)} \left( \frac{q^0}{p^0} p_j - q_j\right) \right\}. 
\]
The terms inside the bracket $\{ \cdot\}$ can be estimated as 
\[
\frac{|p|}{(p^0)^2} = \frac{|p|}{p^0} \frac{1}{p^0}  < \frac{1}{p^0},
\]
and 
\[
\begin{split}
&\frac{\rho +3}{\rho (\rho +1 )(\rho +2)} \frac{1}{p^0} \left| q^0 p - p^0 q \right| \leq c \rho^{-1}  |p-q|   \leq  c \frac{ p^0 q^0 }{  |p-q|}, 
\end{split}
\]
where we use the lower bound for $\rho$,  $\rho \geq \frac{|p-q|^2}{2 p^0 q^0}$,  and also  observe that  
\[
 |q^0 p - p^0 q| \leq  2 p^0 |p-q|,  
\]
and 
\[
\sup_{x \geq 0} \frac{x+3}{(x+1)(x+2)}  \leq 2. 
\]
Consequently, combining the estimate \eqref{EstOfBarLam}, one has 
\begin{equation}\notag
|\nabla_p \bar \Lambda(p, q)|  
%\\
 \leq c \left( \frac{1}{p^0} + \frac{p^0 q^0}{|p-q|} \right)  \left\{  (p^0 q^0)^{1/2} \, |p-q|^{-3} \IN_A(p, q) +   \frac{1}{p^0 q^0} \IN_{A^c}(p, q) \right\}. 
\end{equation}
Note that 
\[
\frac{1}{p^0}\leq 2 \frac{ q^0}{|p-q|} \leq 2 \frac{p^0 q^0}{|p-q|}
\]
simply by 
\[
|p-q| \leq 2 \max\{p^0, q^0\} \leq 2 p^0 q^0.
\]
Thus 
\begin{equation}\label{EstDerBarLam}
\begin{split}
|\nabla_p \bar \Lambda(p, q)|  & \leq c   \left\{   (p^0 q^0)^{3/2} \, |p-q|^{-4} \IN_{A}(p, q)\, +   |p-q|^{-1}\IN_{A^c} (p, q)   \right\}. 
\end{split}
\end{equation}
By symmetry,  $|\nabla_q \bar \Lambda(p, q)|$ satisfies  the same estimate \eqref{EstDerBarLam}.   Then further recall that $\pZ \approx \qZ$ on $A$ as in \eqref{AProperty}.  Thus we have that
\begin{equation}\label{EstDer2BarLam}
|\nabla_p \bar \Lambda(p, q)|   \leq c   \left\{   (q^0)^{3} \, |p-q|^{-4} \IN_{A}(p, q)\, +   |p-q|^{-1}\IN_{A^c} (p, q)   \right\}. 
\end{equation}
Now by a variant of the mean value theorem we {\it claim} that
\begin{multline}\label{MeanVlaue}
|\bar \Lambda(p,q) - \bar \Lambda(\tilde p, q) | \lesssim |p-\tilde{p}|   \bigg[     (q^0)^{3}  |p-q|^{-4} \textbf{1}_A(p, q) + |p-q|^{-1} \textbf{1}_{A^c}(p,q)  
\\
  + (q^0)^{3}  |\tilde p-q|^{-4} \textbf{1}_A( \tilde p, q) + |\tilde p-q|^{-1} \textbf{1}_{A^c}(\tilde   p,q)     \bigg].  
\end{multline}
We will prove \eqref{MeanVlaue} at the end of the proof.

Combining   \eqref{BI_1One} and \eqref{MeanVlaue},  one finally obtains  
\[
\begin{split}
& I_1 = |B(p, q) - B(\tilde p, q)| \\
& \leq c\,  |p - \tilde p | \bigg[\left\{(q^0)^{3} |p-q|^{-3} \IN_A(p, q) +     \IN_{A^c}(p, q)\right\} \\
&\quad \quad  \qquad  +  \left\{ ( q^0)^{3} |\tilde p-q|^{-3} \IN_A(\tilde p, q)\, +    \IN_{A^c}(\tilde p, q) \right\} \\
&\quad \quad   +  \left\{ (p^0 q^0)^{1/2} |p-q|^{-3} \IN_A(p, q)  +  (p^0 q^0)^{-1} \IN_{A^c}(p, q)  \right\} \\
& \quad \quad    +  \left\{ (\tilde p^0 q^0)^{1/2} |\tilde p-q|^{-3} \IN_A(\tilde p, q) +  (\tilde p^0 q^0)^{-1} \IN_{A^c}(\tilde p, q) \right\} \bigg]. 
\end{split}
\]
We can simplify it to be the following 
\begin{equation*}
 I_1 = |B(p, q) - B(\tilde p, q)|  \leq c\,  |p - \tilde p |  \left( \varphi_B^2(p, q) + \varphi_B^2(\tilde p, q) \right) 
\end{equation*}
where  the function $\varphi_B^2$ is defined in \eqref{LipsB}.  By symmetry, one can estimate $I_2$ similarly.

To complete the proof we need to establish  the claim  \eqref{MeanVlaue}.    To do so  let us define a smooth  path $r:[0, 1] \to \mathbb{R}^3$ such that $r(0) = \tilde{p}$, $r(1) = p$, and $|r'(t)| \leq C |p-\tilde{p}|$ for any $t \in [0, 1]$.  We can choose this path around the fixed point $q$, such that $|r(t) - q| \geq c\min \{|p-q|, |\tilde p - q| \} $ for any $t \in [0, 1]$.  We mention that there are many choices of such a path. For instance, if we assume that $|\tilde p - q| = |p -q |=\delta>0$, then we can choose the great circle between $p$ and $\tilde p $ on the sphere $\partial B(q, \delta)$ as the path. Otherwise, if say $0 < \delta_1 = |\tilde p - q|< |p - q| = \delta_2$ then we can choose a smooth path (such as a geodesic) in $B(q, \delta_2) \setminus B(q, \delta_1)$, connecting $p $ and $\tilde p$.  In particular we can choose a path $r(t)$ with constant speed.   Then the total length of the chosen path  is comparable to $|p-\tilde p|$.

Then by the fundamental theorem of calculus, we have that
\begin{multline}\label{upper.est.bd}
 |\bar \Lambda(p,q) - \bar \Lambda(\tilde p, q)  | = \left|\int_0^1 \frac{\ud }{\ud t } \bar \Lambda(r(t), q)  \ud t \right | 
\\
 \leq  C  |p-\tilde{p}| \int_0^1 |\nabla_p \bar \Lambda(r(t), q) | \ud t 
\leq  C |p-\tilde{p}| \max_{t \in [0, 1]} |\nabla_p \bar \Lambda(r(t), q) |
\\
 = C  |p - \tilde p|   |\nabla_p \bar \Lambda(r(t_\star), q) |, 
\end{multline}
where there is a fixed $t_\star \in [0,1]$.  To finish the proof of  \eqref{MeanVlaue}, we do a  case by case analysis of the upper bound in \eqref{upper.est.bd} as follows.  

{\bf Case I}: Suppose that $(p, q) \in A$ (which implies that $(p^0 q^0)^{1/2} \geq |p-q|$) and that $(\tilde p, q) \in A^c$  (which implies that $(\tilde p^0 q^0)^{1/2} \leq |\tilde p - q|$) as in Remark \ref{pointwise.remark}.  
\begin{itemize}
\item If  $|p-q| \leq |\tilde p - q|$ and  $(r(t_\star), q) \in A$, we have that
\[
|\nabla_p \bar \Lambda(r(t_\star), q) | \lesssim (q^0)^{3} |r(t_\star) - q|^{-4}  \leq (q^0)^{3} |p- q|^{-4}. 
\]
This is enough to establish \eqref{MeanVlaue} in this case.  In the following in each case we will justify that the term $|\nabla_p \bar \Lambda(r(t_\star), q) |$ in \eqref{upper.est.bd}  satisfies an upper bound which is equivalent to \eqref{MeanVlaue}. 

\item If  $|p-q| \leq |\tilde p - q|$ and  $(r(t_\star), q) \in A^c$, then 
\[
|\nabla_p \bar \Lambda(r(t_\star), q) |  \lesssim |r(t_\star) - q|^{-1}  \lesssim |p - q|^{-1}   \lesssim (q^0)^{3} |p-q|^{-4},
\] 
where the last inequality is ensured by the fact $(p, q) \in A$ and the pointwise estimates as in Remark \ref{pointwise.remark}.  

\item If $|\tilde p - q | \leq |p - q|$ and $( r(t_\star), q) \in A$, then 
\[
|\nabla_p \bar \Lambda(r(t_\star), q) | \lesssim (q^0)^{3}  |r(t_\star) - q|^{-4} \leq (q^0)^{3}  |\tilde p - q|^{-4}. 
\]
To get a suitable estimate in this case, we will show that 
\[
(q^0)^{3}  |\tilde p - q|^{-4} \lesssim |\tilde p - q|^{-1} + (q^0)^{3}  |p - q|^{-4}. 
\]
To this end if $|\tilde p - q | \geq \frac 1 2 \,  q^0 $, then 
\[
(q^0)^{3}  |\tilde p - q|^{-4} \leq 8 |\tilde p - q|^{-1},
\]
the conclusion then follows. 
Now alternatively assume that $|\tilde p - q| \leq \frac 1 2 q^0$. Using the same technique as in the Remark \ref{pointwise.remark}, one can show that 
\[
\frac{1}{2} q^0 \leq \tilde p^0 \leq \frac{3}{2} q^0. 
\] 
Now since further  $(p, q) \in A$ and $(\tilde p, q) \in A^c$ we have that
\[
\frac{1}{\sqrt{2}} q^0 \leq (\tilde p^0 q^0)^{1/2}  \leq |\tilde p - q| \leq |p - q | \leq (p^0 q^0)^{1/2} \leq \sqrt{3} q^0,
\]
which shows that $(q^0)^{3}  |\tilde p - q|^{-4} \lesssim (q^0)^{3}  |p-q|^{-4 }$ and this establishes the desired estimate in this case.   This holds since 
$
q^0 \approx |\tilde p - q| \approx  |p - q |
$
in this range.

\item If $|\tilde p - q | \leq |p - q|$ and $( r(t_\star), q) \in A^c$, then trivially 
\[
|\nabla_p \bar \Lambda(r(t_\star),q) | \lesssim |r(t_\star) - q|^{-1} \leq |\tilde p - q|^{-1}. 
\] 
This completes the proof of Case 1. 
\end{itemize}

{\bf Case II:} This is the case where $(p, q) \in A^c$  and $(\tilde p, q) \in A$. This case can be treated exactly the same as Case I by symmetry.

{\bf Case III}: If both $(p, q)$ and $(\tilde p, q)$ are in $A$, without loss of generality we assume that $|p - q| \leq |\tilde p - q|$.   The case $|p - q| \geq |\tilde p - q|$ can be handled in exactly the same way as below.
\begin{itemize}

\item If $(r(t_\star), q) \in A$, then of course we have that
\[
|\nabla_p \bar \Lambda(r(t_\star), q) | \lesssim (q^0)^3 |r(t_\star) - q|^{-4 } \leq (q^0)^3 |p - q|^{-4}. 
\]

\item If $(r(t_\star), q) \in A^c$, then  since  $(p, q) \in A$ we have
\[
|\nabla_p \bar \Lambda(r(t_\star), q) | \lesssim |r(t_\star) -  q|^{-1} \leq  |p - q|^{-1}  \leq (p^0 q^0)^{3/2} |p -q|^{-4} \lesssim (q^0)^3 |p - q|^{-4}. 
\]
This completes the estimates for Case III.
\end{itemize}

{\bf Case IV}: If both $(p, q) $ and $(\tilde p, q)$ are in $A^c$, again without loss of generality we can assume that $|p - q| \leq |\tilde p - q|$.  

\begin{itemize}
\item If $(r(t_\star), q) \in A^c$, we can conclude by 
\[
|\nabla_p \bar \Lambda(r(t_\star), q) | \lesssim |r(t_\star) -  q|^{-1} \leq  |p - q|^{-1}. 
\] 

\item Now consider the case that $(r(t_\star), q) \in A$. We need to establish the estimate 
\[
 |\nabla_p \bar \Lambda(r(t_\star), q) | \lesssim  (q^0)^3  |r(t_\star) -q|^{-4}  
 \lesssim (q^0)^3 |p - q|^{-4} \lesssim |p - q|^{-1}. 
\]
In particular if we establish $(q^0)^3 |p - q|^{-4} \lesssim |p - q|^{-1}$ then we are done.  To this end if
 $|p  - q| \geq \frac 1 2 q^0$, then we can conclude by 
\[
(q^0)^3 |p- q|^{-4} \leq 8 |p - q|^{-1}. 
\]
Otherwise, if $|p -q| \leq \frac 1 2 q^0$, as in Case I, we have that 
\[
\frac 1 2  q^0 \leq p^0 \leq \frac 3 2 q^0. 
\]
Combining this with the condition that $(p,q) \in A^c$ we have that
\[
\frac{1}{\sqrt{2}} q^0 \leq (p^0 q^0 )^{1/2} \leq |p - q| \leq \frac{1}{2} q^0, 
\]
which implies that $(q^0)^3 |p-q|^{-4} \approx |p-q|^{-1} \approx 1/q^0$.   This establishes the desired estimate.
\end{itemize}
This completes the proof. 
\end{proof}

In the next proposition, we will  estimate $|\Sigma(p, q) - \Sigma(\tilde{p}, \tilde{q})|^2$. 

\begin{proposition}[Pointwise estimate for $\Sigma$] \label{PropSig} One first has the trivial estimate of $\Sigma$ as 
\begin{equation}\label{TriSig}
\begin{split}
|\Sigma(p, q) - \Sigma(\tilde{p}, \tilde{q})|^2  \leq  & c   \left(\varphi_\Sigma^1(p, q) + \varphi_\Sigma^1 (\tilde p, \tilde q)  \right) 
\end{split}
\end{equation}
where 
\begin{equation}
\label{DefPhiSig1}
\varphi_\Sigma^1(p, q)   \eqdef   \min\{ (p^0)^3, (q^0)^3 \} |p-q|^{-1} \IN_A(p, q) 
%\\
+ \min\{ (p^0)^2, (q^0)^2 \} \IN_{A^c}(p, q), 
\end{equation}
and the Lipschitz estimate of $\Sigma$ as 
\begin{equation}
\label{LipSig}
\begin{split}
&|\Sigma(p, q) - \Sigma(\tilde{p}, \tilde{q})|^2 \leq 2 |\Sigma(p, q) - \Sigma(\tilde{p}, {q})|^2 + 2 |\Sigma( \tilde p, q) - \Sigma(\tilde{p}, \tilde{q})|^2 \\
& \leq c |p- \tilde p|^2 \left( \varphi_\Sigma^2(p, q) + \varphi_\Sigma^2(\tilde p, q) \right)  + c |q- \tilde{q}|^2 \left( \varphi_\Sigma^2(q, \tilde p) + \varphi_\Sigma^2(\tilde q, \tilde p)  \right)  \\
\end{split}
\end{equation}
where 
\begin{equation}
\label{DefPhiSig2} 
\varphi_\Sigma^2(p, q)   \eqdef   \min\{ (p^0)^7, (q^0)^7 \} |p-q|^{-3} \IN_A(p, q) +  (q^0)^5 \IN_{A^c}(p, q). 
\end{equation}
\end{proposition}

\begin{proof}  We first prove \eqref{TriSig}. By the definition of $\Sigma(p, q)$, 
\begin{equation}\label{EstSigProof1}
|\Sigma(p, q)|^2 \leq \Lambda(p, q) |\sigma_S(p, q)|^2 . 
\end{equation}
Now we only need to bound $\sigma_S$, {\em i.e. }
\begin{equation*}
\begin{split}
 \sigma_S = & \begin{bmatrix}
q^0 p_2 -p^0 q_2 & -(q^0 p_3 -p^0 q_3) & 0 \\
-(q^0 p_1 - p^0 q_1)  & 0 & q^0 p_3 -p^0 q_3 \\
0 & q^0 p_1 -p^0 q_1 & -(q^0 p_2 -p^0 q_2)
\end{bmatrix} \\
& \quad \quad - \frac{1}{p^0 +1} (p \times q) \otimes \begin{bmatrix}
p_3 \\p_2 \\p_1
\end{bmatrix}.
\end{split}
\end{equation*}
As we argue before, 
\[
|q^0 p - p^0 q| \leq 2 \min \{ p^0, q^0 \} \, |p-q|. 
\]
Moreover, we recall as in below \eqref{SigmaPi2} we have that 
\[
(p \times q) \otimes \begin{bmatrix}
p_3 \\p_2 \\ p_1 
\end{bmatrix} = \begin{bmatrix}
p_3(p_2 q_3 -p_3 q_2) & p_2 (p_2 q_3 -p_3 q_2) & p_1 (p_2 q_3 -p_3 q_2) \\
p_3(p_3 q_1 -p_1 q_3) & p_2 (p_3 q_1 -p_1 q_3) & p_1 (p_3 q_1 -p_1 q_3) \\
p_3(p_1 q_2 -p_2 q_1) & p_2  (p_1 q_2 -p_2 q_1)  & p_1 (p_1 q_2 -p_2 q_1) \\
\end{bmatrix}. 
\]
Then we have the following estimate
\[
|p_i q_j - p_j q_i |\leq |q_i| |p_j -q_j| + |q_j| |p_i - q_i| \leq 2 |q|\, |p-q|.
\]
By symmetry of $p$ and $q$, 
\[
|p_i q_j - p_j q_i| \leq  2 \min \{|p|, |q| \} |p-q| \leq 2 \min\{p^0, q^0 \}|p-q|.  
\]
which implies that 
\[
\frac{1}{p^0 +1 } |p_k| |p_i q_j - p_j q_i|  \leq 2 \min\{ p^0, q^0\}  |p-q|. 
\]
Therefore 
\begin{equation}\label{SigSBound}
|\sigma_S (p, q)|^2 \leq c\,  \min\{(p^0)^2, (q^0)^2\} \, |p-q|^2.
\end{equation}
Combining with Lemma \ref{EstimateOfLamada} and the estimate \eqref{EstSigProof1},   we  obtain that 
\begin{multline}\notag
|\Sigma(p, q)|^2 
\\
\leq c  \Big(\min\{ (p^0)^3, (q^0)^3 \} |p-q|^{-1} \IN_A(p, q) + \min\{ (p^0)^2, (q^0)^2 \} \IN_{A^c}(p, q) \Big). 
\end{multline}
The bound \eqref{TriSig} follows directly from this estimate. 

\smallskip
Now we proceed to prove \eqref{LipSig}.  By symmetry, it suffices to control $|\Sigma(p, q) - \Sigma(\tilde p, q)|$. The other term has a similar estimate.

The derivative of $\sqrt{\Lambda}$ is given by 
\[
\partial_{p_j} \sqrt{\Lambda(p, q)} = - \frac{1}{2} \sqrt{\Lambda(p, q)} \, \left\{ \frac{1}{(p^0)^2} p_j + \frac{\rho^2 + 2 \rho + 3}{\rho (\rho +1)(\rho +2)} \left( \frac{q^0}{p^0} p_j - q_j \right)\right\}. 
\]
We perform similar estimates to those in the proof of \eqref{EstDerBarLam}.  In particular
$$
\left| \frac{\rho^2 + 2 \rho + 3}{\rho (\rho +1)(\rho +2)} \left( \frac{q^0}{p^0} p_j - q_j \right)\right|
\lesssim
\rho^{-1} \frac{q^0}{p^0}  |p-q|   
\lesssim
 \frac{  (q^0)^2 }{  |p-q|}. 
$$
Here we used that 
\[
 |q^0 p - p^0 q| \leq  2 q^0 |p-q|. 
\]
Then again following the proof of \eqref{EstDerBarLam}, we  estimate $|\nabla_p \sqrt{\Lambda(p, q)}|$ as 
\[
|\nabla_p \sqrt{\Lambda(p, q)}| \leq c \sqrt{\Lambda(p, q)} \, \left( \frac{1}{p^0} + \frac{ (q^0)^2}{|p-q|}\right). 
\]
We combine the estimate for $\sqrt{\Lambda}$ from Lemma \ref{EstimateOfLamada} and $\frac{1}{p^0} \leq 2 \frac{ q^0}{|p-q|}$ to obtain
\begin{multline}\label{EstDerSqrtLam}
 |\nabla_p \sqrt{\Lambda(p, q)}| 
 \\
\lesssim  \, \left\{  \IN_{A}(p, q)\, (q^0)^2  (p^0 q^0)^{1/4} |p-q|^{-5/2} + \IN_{A^c}(p, q) \,  (q^0)^2 \, |p-q|^{-2}\right\}
\end{multline}
We will use this estimate in a moment.

Now consider a function $\alpha(p, q)$, which denotes any entry of the matrix $\sigma_S(p, q)$ from \eqref{matrixS.eq}. Here we just consider the $(1, 1)$ entry of $\sigma_S$, {\em i.e. } 
\[
\alpha(p, q) = q^0  p_2 - p^0 q_2  - \frac{1}{p^0 +1} p_3 \left( p_2 q_3 - p_3 q_2 \right).   
\]
One can directly check that 
\begin{equation}\label{EstEntryDer}
|\nabla_p\, \alpha(p, q)| \leq c  \, q^0,  \quad |\nabla_q \alpha(p, q)| \leq c\,  p^0. 
\end{equation}
The above estimates also hold for any other entry of $\sigma_S$.

By the symmetry between $p$ and $\tilde p$, we have that
\[
\begin{split}
&|\Sigma(p, q) - \Sigma(\tilde p, q) | \\
& \leq  c    \min\{ \sup_\alpha |\alpha(p, q)|, \sup_\alpha |\alpha(\tilde p, q) |\}  |\sqrt{\Lambda(p, q)} - \sqrt{\Lambda(\tilde p, q)}| \\
& + c   \left( \sqrt{\Lambda(p, q)} + \sqrt{\Lambda(\tilde p, q) }\right) \sup_\alpha |\alpha(p, q) - \alpha( \tilde p, q)|. 
\end{split}
\]
Combining the estimate in Lemma \ref{EstimateOfLamada} with \eqref{SigSBound}, \eqref{EstDerSqrtLam} and \eqref{EstEntryDer} and  applying the same variant of the mean value theorem as in the proof of \eqref{MeanVlaue} we obtain 
\[
\begin{split}
&|\Sigma(p, q) - \Sigma(\tilde p, q) | \\
& 
\lesssim
 |p- \tilde{p}| \Bigg[   \left\{ (q^0)^3  (p^0 q^0)^{1/4} |p-q|^{-3/2} \, \IN_A (p, q) + \,  (q^0)^3 |p-q|^{-1}\IN_{A^c}(p,q)  \right\} \\
& \quad \quad +   \left\{  (q^0)^{3} (\tilde p^0 q^0)^{1/4} |\tilde p-q|^{-3/2} \IN_A(\tilde p, q)\,  +(q^0)^3 |\tilde p - q|^{-1} \IN_{A^c}(\tilde p,q) \right\} \\
& \qquad +    \left\{  q^0 (p^0 q^0)^{1/4} |p-q|^{-3/2} \IN_A(p, q) + \, q^0  |p -  q|^{-1}  \IN_{A^c}(p, q) \right\} \\
& \qquad +   \left\{  q^0 (\tilde p^0 q^0)^{1/4} |\tilde p-q|^{-3/2} \IN_A( \tilde p, q) \, +  q^0  |\tilde p -  q|^{-1}  \IN_{A^c}(\tilde p, q) \right\}  \Bigg]. 
\end{split}
\]
Now we use the pointwise estimates in Remark \ref{pointwise.remark} to simplify further as
\[
\begin{split}
&\quad |\Sigma(p, q) - \Sigma(\tilde p, q) |^2  
\\
&
\lesssim
 |p- \tilde{p}|^2 \Bigg[   \left\{ (q^0)^6  (p^0 q^0)^{1/2} |p-q|^{-3} \, \IN_A (p, q) + \,  (q^0)^5 \IN_{A^c}(p,q)  \right\} \\
& \quad \quad +   \left\{  (q^0)^{6} (\tilde p^0 q^0)^{1/2} |\tilde p-q|^{-3} \IN_A(\tilde p, q)\,  +(q^0)^5 \IN_{A^c}(\tilde p,q) \right\}  \Bigg].\\ 
\end{split}
\]
By the symmetry between $p$ and $q$, we also have 
\[
\begin{split}
&\quad |\Sigma( \tilde p, q) - \Sigma(\tilde p,\tilde  q) |^2  \\
& \leq c | q- \tilde{q}|^2  \Bigg[  \left\{  (\tilde p^0)^6( \tilde p^0 q^0)^{1/2} |\tilde p-q|^{-3} \IN_A(\tilde p, q) \, +  (\tilde p^0)^5 \IN_{A^c} (\tilde p, q)\right\}   \\
& \quad  +  \left\{  (\tilde p^0)^6( \tilde p^0 \tilde q^0)^{1/2} |\tilde p- \tilde q|^{-3} \IN_A(\tilde p, \tilde q) \, +  (\tilde p^0)^5 \IN_{A^c} (\tilde p, \tilde q)\right\} \Bigg]. 
\end{split}
\]
The proof of \eqref{LipSig} is complete after combining the last two estimates and using Remark \ref{pointwise.remark} again.  
\end{proof}

\section{Estimates of the integrals} \label{sec:EstInt}

In the previous section we gave a series of pointwise estimates for the relevant quantities in the relativistic Landau equation.  In this section we prove necessary estimates of integrals of these relevant coefficients.   We first recall the following lemma from \cite{MR2718931}:

\begin{lemma}[Lemma 4 in \cite{MR2718931}]\label{lemmaSUP} Let $\alpha \in ( -3, 0]$. There exists a constant $c_\alpha>0$, such that  for all $g \in L^\infty \cap L^1 $ and $\eps \in (0, 1]$, 
\begin{equation}\label{LemIntI}
\sup_{ p \in \mathbb{R}^3}   \int_{\mathbb{R}^3}  |p -q|^\alpha\,  g(q) \ud q   \leq \| g\|_{L^1} + c_\alpha \|g\|_{L^\infty}, 
\end{equation}
\begin{equation}\label{LemIntII}
\int_{\mathbb{R}^3} \int_{\mathbb{R}^3}  |p -q|^\alpha\,  g(p) g(q)  \ud p \ud q  \leq  \left( \| g\|_{L^1} + c_\alpha \|g\|_{L^\infty} \right) \| g\|_{L^1}, 
\end{equation}
\begin{equation}\label{LemIntIII}
 \int_{|\tilde p - q| \leq \eps } |p - q|^\alpha g(q) \ud q \leq c_\alpha \|g \|_{L^\infty} \eps^{3 + \alpha}.  
\end{equation}
We note that the constant $c_\alpha>0$ in \eqref{LemIntIII} is independent of $p$ and ${\tilde p}$.  

Furthermore, there exists a universal constant $c>0$ such that for all $g \in L^\infty \cap L^1$ and for all $\eps \in (0, 1]$, 
\begin{equation}\label{LemIntIV}
 \int_{|p-q| \geq \eps} |p - q|^{-3} g(q) \ud q \leq \| g\|_{L^1} + c \|g \|_{L^\infty} \log(1/\eps). 
\end{equation}
We note that the constant $c>0$ in \eqref{LemIntIV} is independent of $p$.  
\end{lemma}

The proof is standard and hence we omit it. Interested readers can find the complete proof in \cite{MR2718931}.  We further recall Definition \ref{DefPsi} regarding the function $\Psi$ used again below.  We now state two crucial propositions. 

\begin{proposition} \label{PropEstIntegralSingle} Assume that $g \in \PP \cap L^\infty$ and $(q^0)^7 g (q) \in L^\infty \cap L^1$. Then 
\begin{equation}\label{EstSigIntegral}
\int_{\R^3} |\Sigma(p, q) - \Sigma( \tilde p, q)|^2 g(q) \ud q \leq C(g) \Psi( |p - \tilde{p}|^2), 
\end{equation}
\begin{equation} \label{EstBIntegral}
\int_{\R^3} |B(p, q) - B( \tilde p, q)| g(q) \ud q \leq C(g) \Psi( |p - \tilde{p}|), 
\end{equation}
where
\begin{equation} \label{const.def.Integral}
C(g) =c\left(   \| g \|_{L_7^\infty} +  \| g \|_{L_7^1} +1\right),
\end{equation}
with $c>0$ is a universal constant. 
\end{proposition}

\begin{proof}
We define by $I$ the left hand side of \eqref{EstSigIntegral}. Recall Proposition \ref{PropSig} which shows that
\begin{multline}\notag
|\Sigma(p, q) - \Sigma(\tilde p, q)|^2 \lesssim
 \min\left\{ \varphi_\Sigma^1(p, q) + \varphi_\Sigma^1( \tilde p, q), \right.
 \\
\left. |p - \tilde p|^2 ( \varphi_\Sigma^2(p, q) + \varphi_\Sigma^2( \tilde p, q) ) \right\} . 
\end{multline}
where   
\[
\varphi_\Sigma^1(p, q) \leq  (q^0 )^3 |p-q|^{-1} \IN_A(p, q) +  \, (q^0)^2 \IN_{A^c}(p, q), 
\]
\[
\varphi_\Sigma^2(p, q) \leq  (q^0)^7  |p-q|^{-3} \, \IN_A (p, q) + \,  (q^0)^6 \IN_{A^c}(p,q). 
\]
Hence 
\begin{multline}\notag
I \lesssim
 \IN_{\{|p - \tilde p| \geq 1\}} \int_{\R^3 } \left( \varphi_\Sigma^1(p, q) + \varphi_\Sigma^1( \tilde p, q)  \right) g(q) \ud q  
 \\
 +   \IN_{\{|p - \tilde p| \leq  1\}} |p- \tilde p|^2  \int_{\R^3 } \IN_{\{ |p - q| \geq  |p - \tilde p|^2,  |\tilde p  - q| \geq  |p - \tilde p|^2 \}}  \left( \varphi_\Sigma^2(p, q) + \varphi_\Sigma^2( \tilde p, q)  \right) g(q) \ud q 
 \\
 +   \IN_{\{|p - \tilde p| \leq  1\}} \int_{\R^3 }  \IN_{ \{ |p-q| \leq |p - \tilde p|^2 \} }  \left( \varphi_\Sigma^1(p, q) + \varphi_\Sigma^1( \tilde p, q)  \right) g(q) \ud q  
 \\
 +  \IN_{\{|p - \tilde p| \leq  1\}} \int_{\R^3 }  \IN_{ \{ | \tilde p-q| \leq |p - \tilde p|^2 \} }  \left( \varphi_\Sigma^1(p, q) + \varphi_\Sigma^1( \tilde p, q)  \right) g(q) \ud q  
 \\
 \lesssim \left( I_1 + I_2 + I_3 + I_4 \right). 
\end{multline}
First, by   \eqref{LemIntI} (with $\alpha =-1$),  for any $p \in \R^3$, 
\[
\begin{split}
 & \int_{\R^3} \varphi_\Sigma^1(p, q) g(q) \ud q \leq  \int_A  |p-q|^{-1} \left( (q^0)^3 g(q) \right) \ud q  + \int_{A^c} (q^0)^2 g(q) \ud q  \\
 & \leq c  \left( \| g \|_{L_3^\infty} +  \| g \|_{L_3^1} \right) \leq C(g),
 \end{split}
\]
which using the symmetry between $p$ and ${\tilde p}$ implies that
\[
I_1 \leq c \IN_{\{|p - \tilde p| \geq 1\}}  C(g)  \leq C(g) \Psi(|p- \tilde p|^2). 
\]
Next, using \eqref{LemIntIV} with $\eps = |p - \tilde p|^2$, 
\[
\begin{split}
I_2 & \leq \IN_{\{|p - \tilde p| \leq  1\}}  |p - \tilde p|^2 \bigg( \int_{\R^3 } \IN_{\{ |p - q| \geq  |p - \tilde p|^2 \}}  \varphi_\Sigma^2(p, q)  g(q) \ud q  \\
& + \int_{\R^3 } \IN_{\{ | \tilde p - q| \geq  | p - \tilde p|^2 \}}  \varphi_\Sigma^2( \tilde p, q)  g(q) \ud q \bigg) \\
&  \leq  2 \,  \IN_{\{|p - \tilde p| \leq  1\}}  |p - \tilde p|^2 \bigg( \int_{ |p - q| \geq  |p - \tilde p|^2 } |p -q|^{-3}  (q^0)^7 g(q) \ud q  \\
& + \int_{ |p - q| \geq  |p - \tilde p|^2 } (q^0)^6 g(q) \ud q \bigg)\\
& \leq C(g) \IN_{\{|p - \tilde p| \leq  1\}}  |p - \tilde p|^2 \left( 1 - \log |p- \tilde p|^2 \right)  \\
& \leq C(g) \Psi(|p - \tilde p|^2), 
\end{split}
\]
noting that $p$ and $\tilde p$ are exchangeable in the 2nd inequality. 

Finally, from \eqref{LemIntIII} with $\alpha=-1$ and $\eps= |p - \tilde p|^2$, we have 
\[
\begin{split}
I_3 + I_4 \leq &  2\,  \IN_{ \{|p-\tilde p | \leq 1 \}} \bigg( \int_{ |p- q| \leq |p - \tilde p|^2 } \left( |p-q|^{-1} + |\tilde p -q|^{-1} \right) (q^0)^3 g(q)  \ud q  \\
 & \qquad \quad +    \int_{|p-q| \leq |p - \tilde p|^2} (q^0)^2 g(q) \ud q \bigg) \\
 & \leq C(g) \IN_{\{ |p-\tilde p | \leq 1\} } (|p - \tilde p|^2)^{3-1} \leq C(g) \Psi(|p - \tilde p|^2).  
\end{split}
\]
Collecting these estimates yields \eqref{EstSigIntegral}.

We turn to the estimate of  \eqref{EstBIntegral}.  We then denote by $J$ the left hand side of \eqref{EstBIntegral}.   Further recall  Proposition \ref{PropEstB}, which says that
\begin{equation}\notag
|B(p, q) - B(\tilde p, q) |   \leq c \min\{ \varphi_B^1(p, q) + \varphi_B^1( \tilde p, q), 
   |p-\tilde p| (\varphi_B^2(p, q) + \varphi_B^2(\tilde p, q) ) \}, 
\end{equation}
with 
\[
\varphi_B^1(p, q) \leq |p-q|^{-2} q^0  \IN_A(p, q) +   \IN_{A^c}(p, q), 
\]
\[
\varphi_B^2(p, q) \leq |p-q|^{-3} (q^0)^3  \IN_A(p, q) +   \IN_{A^c}(p, q).
\]
Then one can proceed to estimate $J$ as what we have done  for $I$
\begin{multline}\notag
J  \lesssim \IN_{\{|p - \tilde p| \geq 1\}} \int_{\R^3 } \left( \varphi_B^1(p, q) + \varphi_B^1( \tilde p, q)  \right) g(q) \ud q  \\
 +   \IN_{\{|p - \tilde p| \leq  1\}} |p- \tilde p|  \int_{\R^3 } \IN_{\{ |p - q| \geq  |p - \tilde p|^2,  |\tilde p  - q| \geq  |p - \tilde p|^2 \}}  \left( \varphi_B^2(p, q) + \varphi_B^2( \tilde p, q)  \right) g(q) \ud q \\
 +   \IN_{\{|p - \tilde p| \leq  1\}} \int_{\R^3 }  \IN_{ \{ |p-q| \leq |p - \tilde p|^2 \} }  \left( \varphi_B^1(p, q) + \varphi_B^1( \tilde p, q)  \right) g(q) \ud q  \\
 +  \IN_{\{|p - \tilde p| \leq  1\}} \int_{\R^3 }  \IN_{ \{ | \tilde p-q| \leq |p - \tilde p|^2 \} }  \left( \varphi_B^1(p, q) + \varphi_B^1( \tilde p, q)  \right) g(q) \ud q  \\
 \lesssim \left( J_1 + J_2 + J_3 + J_4 \right). 
\end{multline}
Using \eqref{LemIntI} for $\alpha= -2$, we obtain 
\[
J_1 \leq c \left( \|g\|_{L^1} + \| q^0 g\|_{L^1} + \|q^0 g\|_{L^\infty} \right) \IN_{\{ |p - \tilde p| \geq 1 \}} \leq C(g) \Psi(|p -\tilde p|). 
\]
Again, \eqref{LemIntIV} with $\eps= |p - \tilde p|^2 $ yields
\[
\begin{split}
J_2 & \leq \IN_{\{|p - \tilde p| \leq  1\}}  |p - \tilde p| \bigg( \int_{\R^3 } \IN_{\{ |p - q| \geq  |p - \tilde p|^2 \}}  \varphi_B^2(p, q)  g(q) \ud q  \\
& + \int_{\R^3 } \IN_{\{ | \tilde p - q| \geq  | p - \tilde p|^2 \}}  \varphi_B^2( \tilde p, q)  g(q) \ud q \bigg) \\
&  \leq  2\,  \IN_{\{|p - \tilde p| \leq  1\}}  |p - \tilde p| \bigg( \int_{ |p - q| \geq  |p - \tilde p|^2 } |p -q|^{-3}  (q^0)^3 g(q) \ud q  \\
& + \int_{ |p - q| \geq  |p - \tilde p|^2 }  g(q) \ud q \bigg)\\
& \leq c \left(  \| g \|_{L^1} + \| (q^0)^3 g \|_{L^1} + \| (q^0)^3 g \|_{L^\infty} \right) \IN_{\{|p - \tilde p| \leq  1\}}  |p - \tilde p| \left( 1 - \log |p- \tilde p|^2 \right)  \\
& \leq C(g) \IN_{\{|p - \tilde p| \leq  1\}}  |p - \tilde p| \left( 1 - \log |p- \tilde p| \right) \\
& \leq C(g) \Psi(|p - \tilde p|). 
\end{split}
\]
Finally, we use \eqref{LemIntIII} with $\eps= |p - \tilde p|^2$ and $\alpha = -2$
\[
J_3 + J_4 \leq c \| q^0 g\|_{L^\infty} |p - \tilde p|^2 \IN_{\{ |p - \tilde p| \leq 1\}} \leq C(g) \Psi( |p - \tilde p|). 
\]
This completes the proof.
\end{proof}

In the next proposition we prove bounds for the integrals of differences of the quantity $B$ in  \eqref{BiDef}.

\begin{proposition}\label{EstDoubleB} Assume that $g, \tilde g \in \PP \cap L^\infty$ and also $(q^0)^3 g, (q^0)^3 \tilde g \in L^\infty \cap L^1$. Then for  any two couplings $Q, R \in \ProbJ(g, \tilde g)$, we have that
\begin{multline}\label{EstDoubleBEqu}
 \int_{\R^3 \times \R^3} \int_{\R^3 \times \R^3} |p- \tilde p| \cdot |B(p, q) - B(\tilde p, \tilde q)| Q(\ud p, \ud \tilde p) R(\ud q, \ud \tilde q) 
 \\
\leq \widetilde C(g, \tilde g)  \Psi\left( \int_{\R^3 \times \R^3} |p - \tilde p|^2 Q(\ud p, \ud \tilde p) \right) 
\\
+C(g, \tilde g)  \Psi\left( \int_{\R^3 \times \R^3} |q - \tilde q|^2 R(\ud q, \ud \tilde q) \right), 
\end{multline}
where using the constant defined in \eqref{const.def.Integral} we have that
\begin{align*}
\widetilde C(g, \tilde g) \eqdef C(g) + C(\tilde g) & = c \left( \| (q^0)^3 g \|_{L^\infty} 
+   \| q^0 g\|_{L^1}+1 \right) \\&  
+ c \left( \| (q^0)^3 \tilde g \|_{L^\infty} 
+ \| q^0 \tilde g\|_{L^1} +1\right),
\end{align*}
where $c>0$ is a universal constant. 
\end{proposition}

\begin{proof}
We denote by $K$ the left hand side of \eqref{EstDoubleBEqu}, e.g.
$$
K 
\eqdef 
 \int_{\R^3 \times \R^3} \int_{\R^3 \times \R^3} |p- \tilde p| \cdot |B(p, q) - B(\tilde p, \tilde q)| Q(\ud p, \ud \tilde p) R(\ud q, \ud \tilde q), 
$$
and we denote  
\[
\delta(p, \tilde p, q, \tilde q) \eqdef |p - \tilde p | \cdot |B(p, q) - B(\tilde p, \tilde q)|. 
\]
By Proposition \ref{PropEstB}, we have
\begin{align*}
\delta \leq  & c \left( |p - \tilde p| + |q- \tilde q | \right) \min \Big\{ \varphi_B^1(p, q) +\varphi_B^1(\tilde p, \tilde q),   \\
& |p - \tilde p| (\varphi_B^2(p, q) + \varphi_B^2(\tilde p, q)) + |q-\tilde q| (\varphi_B^2(q, \tilde p) + \varphi_B^2(\tilde q, \tilde p)) \Big\}, 
\end{align*}
where $\varphi_B^1, \varphi_B^2$ are defined in \eqref{TriB} and \eqref{LipsB} respectively. We decompose the integral $K$ into integrals on different regions. For instance, if $|p-\tilde p| + |q- \tilde q| \geq 1$, we simply use the trivial bound 
\[
\delta \leq c \left( |p- \tilde p | + |q- \tilde q| \right) \left( \varphi_B^1(p, q) +\varphi_B^1(\tilde p, \tilde q)\right). 
\]
Otherwise, outside a small neighborhood of the critical singularity $|p-q|^{-3}$ appeared in $\varphi_B^2$, we use the second estimate in the minimmum above to obtain
\begin{align*}
\delta \leq c  & \left( |p - \tilde p| + |q- \tilde q | \right)  \Big( |p - \tilde p| (\varphi_B^2(p, q) + \varphi_B^2(\tilde p, q)) + \\ &|q-\tilde q| (\varphi_B^2(q, \tilde p)  + \varphi_B^2(\tilde q, \tilde p)) \Big) \\
\leq c & \left( |p -\tilde p|^2 + |q- \tilde q|^2 \right) \left( \varphi_B^2(p, q) + \varphi_B^2(\tilde p, q) + \varphi_B^2(\tilde q, \tilde p) \right). 
 \end{align*}
Near the singularity, one still uses the trivial bound. Indeed, we write 
\[
\delta_1(p, \tilde p, q, \tilde q) =   \IN_{ \{ |p-\tilde p| + |q - \tilde q| \geq 1\}}  \left( |p- \tilde p | + |q- \tilde q| \right) \left( \varphi_B^1(p, q) +\varphi_B^1(\tilde p, \tilde q)\right), 
\]
\begin{align*}
\delta_2(p, &\tilde p, q, \tilde q) =  \IN_{ \{ |p-\tilde p| + |q - \tilde q| \leq 1\}} \IN_{\{ |p-q| \geq |p - \tilde p|^4, |\tilde p -q| \geq |p- \tilde p|^4, |\tilde p - \tilde q| \geq |p - \tilde p|^4 \}}  \\
& \quad   \quad \quad \cdot   \left( |p -\tilde p|^2 \right)  \left( \varphi_B^2(p, q) + \varphi_B^2(\tilde p, q) + \varphi_B^2(\tilde q, \tilde p) \right) , 
\end{align*}
\begin{align*}
\delta_3(p, &\tilde p, q, \tilde q) =     \IN_{ \{ |p-\tilde p| + |q - \tilde q| \leq 1\}} \IN_{\{ |p-q| \geq |q - \tilde q|^4, |\tilde p -q| \geq |q- \tilde q|^4, |\tilde p - \tilde q| \geq |q - \tilde q|^4 \}} \\
& \quad   \quad \quad \cdot   \left( |q -\tilde q|^2 \right)  \left( \varphi_B^2(p, q) + \varphi_B^2(\tilde p, q) + \varphi_B^2(\tilde q, \tilde p) \right), 
\end{align*}
and further $\delta_4$ through $\delta_6$ are defined as 
\[
\delta_4(p, \tilde p, q, \tilde q) =  \IN_{ \{ |p-\tilde p| + |q - \tilde q| \leq  1\}}  \IN_{\{ |p-q| \leq |p- \tilde p|^4 \}}  \left( \varphi_B^1(p, q) +\varphi_B^1(\tilde p, \tilde q)\right), 
\]
\[
\delta_5(p, \tilde p, q, \tilde q) =  \IN_{ \{ |p-\tilde p| + |q - \tilde q| \leq  1\}}  \IN_{\{ | \tilde p-q| \leq |p- \tilde p|^4 \}}  \left( \varphi_B^1(p, q) +\varphi_B^1(\tilde p, \tilde q)\right), 
\]
\[
\delta_6(p, \tilde p, q, \tilde q) =  \IN_{ \{ |p-\tilde p| + |q - \tilde q| \leq  1\}}  \IN_{\{ | \tilde p- \tilde q| \leq |p- \tilde p|^4 \}}  \left( \varphi_B^1(p, q) +\varphi_B^1(\tilde p, \tilde q)\right), 
\]
and $\delta_7, \delta_8, \delta_9$ can be defined similarly as $\delta_4, \delta_5, \delta_6$ but now  the 2nd indicator functions are  $\IN_{\{|p-q| \leq |q- \tilde q|^4 \}}$,  $\IN_{\{ | \tilde p-q| \leq |q- \tilde q|^4 \}}  $ and $\IN_{\{ | \tilde p- \tilde q| \leq | q- \tilde q|^4 \}}  $ respectively.  Indeed, we have that
\[
\delta_7(p, \tilde p, q, \tilde q) =  \IN_{ \{ |p-\tilde p| + |q - \tilde q| \leq  1\}}  \IN_{\{ |p-q| \leq |q- \tilde q|^4 \}}  \left( \varphi_B^1(p, q) +\varphi_B^1(\tilde p, \tilde q)\right), 
\]
\[
\delta_8(p, \tilde p, q, \tilde q) =  \IN_{ \{ |p-\tilde p| + |q - \tilde q| \leq  1\}}  \IN_{\{ | \tilde p-q| \leq |q- \tilde q|^4 \}}  \left( \varphi_B^1(p, q) +\varphi_B^1(\tilde p, \tilde q)\right), 
\]
\[
\delta_9(p, \tilde p, q, \tilde q) =  \IN_{ \{ |p-\tilde p| + |q - \tilde q| \leq  1\}}  \IN_{\{ | \tilde p- \tilde q| \leq |q- \tilde q|^4 \}}  \left( \varphi_B^1(p, q) +\varphi_B^1(\tilde p, \tilde q)\right). 
\]
We will decompose $K$ in terms of all of these decompositions above.

Thus we obtain
\[
K \leq c \sum_{i=1}^9 K_i, 
\]
where
\[
K_i = \int_{\R^3 \times \R^3}  \int_{\R^3 \times \R^3} \delta_i(p, \tilde p, q, \tilde q ) Q( \ud p, \ud \tilde p) R( \ud q, \ud \tilde q). 
\]
We will estimate each term individually

\smallskip
First, we have the estimate
\begin{align*}
\delta_1(p, \tilde p, q, \tilde q) & \leq    \left( |p- \tilde p | + |q- \tilde q| \right)^2 \left( \varphi_B^1(p, q) +\varphi_B^1(\tilde p, \tilde q)\right) \\
& \leq 2 \left( |p - \tilde p|^2 + |q - \tilde q|^2 \right) \left( \varphi_B^1(p, q) +\varphi_B^1(\tilde p, \tilde q)\right). 
\end{align*}
Consequently, 
\begin{align*}
K_1 & \leq 2 \int_{\R^3 \times \R^3}  |p-\tilde p|^2 Q(\ud p, \ud \tilde p)  \int_{\R^3 \times \R^3} \left( \varphi_B^1(p, q) +\varphi_B^1(\tilde p, \tilde q) \right) R( \ud q, \ud \tilde q) \\ 
& +   2 \int_{\R^3 \times \R^3}  |q-\tilde q|^2 R(\ud q, \ud \tilde q)  \int_{\R^3 \times \R^3} \left( \varphi_B^1(p, q) +\varphi_B^1(\tilde p, \tilde q) \right) Q(\ud p, \ud \tilde p )
\end{align*}
Since $R$ has marginals  $g$ and $\tilde g$,   we have the estimate
\begin{align*}
 & \int_{\R^3 \times \R^3} \varphi_B^1(p, q) R(\ud q, \ud \tilde q) = \int_{\R^3 \times \R^3} \varphi_B^1(p, q) g(q) \ud q  \\
 & \leq \int_{\R^3 \times \R^3} q^0 |p-q|^{-2} \IN_A(p, q) g(q) \ud q + \int_{\R^3 \times \R^3}  \IN_{A^c}(p, q) g(q) \ud q \\
& \leq   C \left( \|q^0 g\|_{L^\infty} + \| q^0 g\|_{L^1} \right) + \|g\|_{L^1} \leq C(g),  
\end{align*}
which is deduced from \eqref{LemIntI} with $\alpha=-2$. Similarly, 
\[
\int_{\R^3 \times \R^3} \varphi_B^1(\tilde p, \tilde q) R( \ud q, \ud \tilde q) \leq C( \tilde g), 
\]
and
\[
\int_{\R^3 \times \R^3} \left( \varphi_B^1(p, q) +\varphi_B^1(\tilde p, \tilde q) \right) Q(\ud p, \ud \tilde p ) \leq C(g)  + C(\tilde g). 
\]
Thus we obtain
\begin{multline*}
K_1 \leq  C(g, \tilde g)   \left\{   \int_{\R^3 \times \R^3}  |p-\tilde p|^2 Q(\ud p, \ud \tilde p) +  \int_{\R^3 \times \R^3}  |q-\tilde q|^2 R(\ud q, \ud \tilde q) \right\} 
\\
 \leq  C(g, \tilde g)    \Psi\left( \int_{\R^3 \times \R^3}  |p-\tilde p|^2 Q(\ud p, \ud \tilde p) \right)
 \\
  + C(g, \tilde g)   \Psi\left( \int_{\R^3 \times \R^3}  |q-\tilde q|^2 R(\ud q, \ud \tilde q) \right),
\end{multline*}
by simply using that $x \leq \Psi(x)$ for any $x \geq 0$. 

Now we estimate $K_2$. Note that 
\begin{align*}
\delta_2(p, \tilde p, q, \tilde q)  & \leq  \IN_{ \{ |p-\tilde p| \leq 1\}} \IN_{\{ |p-q| \geq |p - \tilde p|^4\}}    |p -\tilde p|^2  \varphi_B^2(p, q)\\
&   +  \IN_{ \{ |p-\tilde p| \leq 1\}} \IN_{\{ |\tilde p-q| \geq |p - \tilde p|^4\}}    |p -\tilde p|^2  \varphi_B^2( \tilde p, q) \\
& + \IN_{ \{ |p-\tilde p| \leq 1\}} \IN_{\{ | \tilde p- \tilde q| \geq |p - \tilde p|^4\}}    |p -\tilde p|^2  \varphi_B^2(\tilde q, \tilde p) \\
&  \eqdef  \sum_{i=1}^3\delta_{2, i}(p, \tilde p, q, \tilde q).  
\end{align*}
Set $K_2 = \sum_{i=1}^3 K_{2, i}$, with $K_{2, i} =\int_{\R^3 \times \R^3} \int_{\R^3 \times \R^3} \delta_{2, i}\,  Q(\ud p, \ud \tilde p) R(\ud q, \ud \tilde q)$. 
Hence, \eqref{LemIntIV} with $\eps= |p- \tilde p|^4$ yields 
\begin{multline*}
 K_{2, 1} \leq \int_{\R^3 \times \R^3} Q(\ud p, \ud \tilde p) \IN_{ \{ |p-\tilde p| \leq 1\}} |p - \tilde p|^2  \int_{ |p -q| \geq |p - \tilde p|^4} \varphi_B^2(p, q) R(\ud q, \ud \tilde q) \\
 \leq   \int_{\R^3 \times \R^3} Q(\ud p, \ud \tilde p) \IN_{ \{ |p-\tilde p| \leq 1\}} |p - \tilde p|^2   \\
 \cdot \left(  \int_{ |p -q| \geq |p - \tilde p|^4} |p-q|^{-3} (q^0)^3 g(q) \ud q  +  \int_{ |p -q| \geq |p - \tilde p|^4}  g(q)  \ud q \right), 
\end{multline*}
which leads to
\begin{multline*}
 K_{2, 1}
 \leq \int_{\R^3 \times \R^3} Q(\ud p, \ud \tilde p) \IN_{ \{ |p-\tilde p| \leq 1\}} |p - \tilde p|^2 
\\
 \quad \cdot \left( \| (q^0)^3 g\|_{L^1} + c \|(q^0)^3 g\|_{L^\infty} \log \frac{1}{|p-\tilde p|^4}  + \|g \|_{L^1}\right). 
\end{multline*}
Thus we conclude that
\begin{multline*}
 K_{2, 1} \leq C(g)  \int_{\R^3 \times \R^3}  Q(\ud p, \ud \tilde p) \IN_{ \{ |p-\tilde p| \leq 1\}} |p - \tilde p|^2 \left(1- \log |p-\tilde p|^2 \right) \\
  \leq C(g)  \int_{\R^3 \times \R^3} \Psi(|p - \tilde p|^2 ) Q( \ud p, \ud \tilde p). 
\end{multline*}
Jensen's inequality, in particular  \eqref{Jensen}, gives us 
\[
K_{2, 1} \leq C(g) \Psi\left( \int_{\R^3 \times \R^3} |p - \tilde p|^2 Q(\ud p, \ud \tilde p) \right). 
\]
The terms $K_{2, 2}$ and $K_{2, 3}$ can be bounded similarly by symmetry. Moreover, $K_3$ can be bounded symmetrically also just by exchanging the role of $p\  ( \tilde p)$ and $q \ (\tilde q)$. 

Now we proceed to estimate $K_4, K_5$ and $K_6$. Note that $K_7, K_8$ and $K_9$ can be bounded similarly to $K_4, K_5$ and $K_6$ . Note further that 
\[
\delta_4(p, \tilde p, q, \tilde q) \leq  \IN_{ \{ |p-\tilde p|  \leq  1\}}  \IN_{\{ |p-q| \leq |p- \tilde p|^4 \}}  \left( \varphi_B^1(p, q) +\varphi_B^1(\tilde p, \tilde q)\right). 
\]
Hence 
\[
\begin{split}
K_4 \leq \int_{\R^3 \times \R^3 } Q( \ud p, \ud \tilde p) &\IN_{|p - \tilde p| \leq 1} \bigg( \int_{|p-q| \leq |p - \tilde p|^4 } \varphi_B^1(p, q) g(q) \ud q  \\
 & +  \int_{|p-q| \leq |p - \tilde p|^4 } \varphi_B^1(\tilde p,  \tilde q) R( \ud q, \ud \tilde q) \bigg).  \\
\end{split}
\]
For the 1st part, using \eqref{LemIntIII} with $\alpha =0$, $\alpha = -2$ and $\eps= |p- \tilde p|^4$, 
\begin{align*}
& \int_{|p-q| \leq |p - \tilde p|^4 } \varphi_B^1(p, q) g(q) \ud q  \\ & \leq  \int_{|p-q| \leq |p - \tilde p|^4 } \left( q^0 |p-q|^{-2} \IN_A(p, q) + \IN_{A^c}(p, q)    \right) g(q) \ud q \\
& \leq  c \|q^0 g \|_{L^\infty} |p- \tilde p|^4  + c \| g\|_{L^\infty} |p - \tilde p|^{12} \leq C(g) |p- \tilde p|^2.
\end{align*}
For the 2nd part, it is a little bit more complicated for instance we must single out variable $\tilde  q $ first by H\"older inequality. Indeed, we have
\begin{align*}
& \int_{|p-q| \leq |p - \tilde p|^4 } \varphi_B^1(\tilde p,  \tilde q) R( \ud q, \ud \tilde q)  \\
& \leq \int_{|p-q| \leq |p - \tilde p|^4 } |\tilde p - \tilde q|^{-2} \tilde{q}^0 R( \ud q, \ud \tilde q) +  \int_{|p-q| \leq |p - \tilde p|^4 } R( \ud q, \ud \tilde q), 
\\
& \leq \left( \int_{|p-q| \leq |p - \tilde p|^4 }  R( \ud q, \ud \tilde q) \right)^{1/5}  \left( \int_{\R^3 \times \R^3}  |\tilde p - \tilde q|^{- 5/2} (\tilde q^0)^{5/4} R( \ud q, \ud \tilde q) \right)^{4/5} \\
& \quad + \int_{|p-q| \leq |p - \tilde p|^4 } R( \ud q, \ud \tilde q) \\
& \leq c \|g\|_{L^\infty}^{1/5} |p - \tilde p |^{12/5} \left(  c \| (\tilde q^0)^{5/4} \tilde g \|_{L^\infty}  +  \| (\tilde q^0)^{5/4} \tilde g \|_{L^1} \right)^{4/5} +  c \|g\|_{L^\infty} |p - \tilde p |^{12}\\
& \leq \widetilde C(g, \tilde g)  | p - \tilde p |^2.
\end{align*}
Above we used \eqref{LemIntI} with $\alpha = 5/2$.
Combining those two parts, one has 
\[
\begin{split}
& K_4 \leq \widetilde C(g, \tilde g) \int_{\R^3 \times \R^3 } Q( \ud p, \ud \tilde p)  \IN_{|p - \tilde p| \leq 1} |p - \tilde p|^2  \\
& \leq \widetilde C(g, \tilde g) \int_{\R^3 \times \R^3 } \Psi\left( |p - \tilde p|^2 \right)Q(\ud p, \ud \tilde p) \\
& \leq \widetilde C(g, \tilde g) \Psi\left(\int_{\R^3 \times \R^3 } |p - \tilde p|^2 Q(\ud p, \ud \tilde p) \right) . 
\end{split}
\]
All the remaining terms can be estimated similarly because of symmetry. This completes the proof. 
\end{proof}

\section{A generalized Gronwall inequality} \label{gronwall.section}

In this section we recall a known generalization of the Gronwall lemma, which will be crucial to conclude our uniqueness argument. 

\begin{lemma} \label{GronwallPsi}Let $T >0$, $\gamma, \, \rho: [0, T] \to [0, \infty)$ and $\gamma \in L^1([0, T])$, $\rho \in L^\infty([0, T])$. Assume further  that 
\[
\rho(t) \leq \rho(0) + \int_0^t \gamma(s) \Psi(\rho(s)) \ud s. 
\]
where we recall  $\Psi$ which is defined in Definition \ref{DefPsi}. \\
i) If $\rho(0) =0$, then $\rho(t) = 0$ for all $t \in [0, T]$. \\
ii) If $\rho(0) >0$, then 
\[
\int_{\rho(0)}^{\rho(t) } \frac{1}{\Psi(y)} \ud y \leq \int_0^t \gamma(s) \ud s, \quad \mbox{for any}\  t \in [0, T]. 
\]
\end{lemma}

We include the proof of this lemma for the sake of completeness.

\begin{proof} Since $\Psi$ is increasing,  the upper bound for $\rho(t)$ is given by the solutions (if any) to the integral equation
\begin{equation}\label{GronIne}
\rho(t) = \rho(0) + \int_0^t \gamma(s) \Psi(\rho(s)) \ud s. 
\end{equation}
which is equivalent  to the Cauchy problem of the  differential equation 
\[
\rho'(t) = \gamma(t) \Psi(\rho(t)), \quad \rho(t=0) = \rho(0). 
\]
Note that if $\rho(0) >0$, then  $\Psi(x)$ is locally Lipschitz on $x>0$. The Cauchy problem above actually has a  unique solution at least in a small interval $[0, \tau]$, by the Cauchy-Lipschitz theory. Indeed, we have that
\[
\int_{\rho(0) }^{\rho(t) } \frac{\ud y  }{ \Psi(y ) } = \int_0^t \gamma (s) \ud  s. 
\]
Note that $\rho$ is non-decreasing since $\gamma$ is non-negative. So we can keep extending  this solution $\rho = \rho(t)$ to any $t \leq T$, since $ \sup_{x \geq \rho(0)}\Psi'(x) \leq C $.   Going back to the inequality in this lemma, we obtain part ii) of this lemma. 

Now we prove part i). We can follow the proof of part ii), for instance the upper bound for $\rho(t)$ is again given by the solution to the integral equation 
\[
\rho(t) = \int_0^t \gamma(s) \Psi(\rho(s) ) \ud s. 
\]
Let $\bar t = \sup\{ t \in [0, T] \vert \rho(t) =0\}$. If $\bar t =T$, then the proof is finished. 

We therefore assume that $\bar t < T$ and prove i) by contradiction. So we pick a sequence of times $t_n > t_{n+1} \to \bar t$ as $n \to \infty$, with each $ \rho(t_n) >0$ but $\rho(t_n) > \rho(t_{n+1}) \to 0$. By the proof of part ii), one has the estimate 
\[
\int_{\rho(t_n)}^{\rho(T) } \frac{1}{\Psi(y) } \ud y \leq \int_{t_n}^T \gamma(s) \ud s \leq \int_0^t \gamma(s) \ud s < \infty. 
\]
However, with fixed $\rho(T) >0$ but sending $\rho(t_n) \to 0$, 
\[
\int_{\rho(t_n)}^{\rho(T) } \frac{1}{\Psi(y) } \ud y \to +\infty, \quad \mbox{as }\  n \to \infty. 
\]
 This is a contradiction. This completes the proof of part i). 
\end{proof}

\section{The main integral inequality }\label{integral.ineq.section}

This  section is devoted to the proof Proposition \ref{PropIntegral}.  To this end we will use the following proposition.

\begin{proposition}[Uniqueness of the coupled SDEs \eqref{SDE1} and \eqref{SDE2}] \label{PropUniquenessSDEs}  There exists a unique pair $(P_t)_{t \in [0, T]}$, $(\tilde P_t)_{t \in [0, T]}$ of continuous $(\mathcal{F}_t)_{t \in [0, T]}-$adapted processes solving the coupled SDEs \eqref{SDE1} and \eqref{SDE2}. In particular, for any $t \in [0, T]$, $\mbox{Law}(P_t) = F_t$ and $\mbox{Law}(\tilde P_t) = \tilde F_t$. 
\end{proposition}

We postpone the proof of Proposition \ref{PropUniquenessSDEs} to later.   We will first   prove Proposition \ref{PropIntegral}, assuming that  Proposition \ref{PropUniquenessSDEs} holds.

\begin{proof}[Proof of Proposition \ref{PropIntegral}]
By Proposition \ref{PropUniquenessSDEs}, $\mbox{Law}(P_t) = F_t$ and $\mbox{Law}(\tilde P_t) = \tilde F_t$, which allows us to conclude that 
\[
\mathcal{W}_2^2(F_t, \tilde F_t) \leq \mathbb{E}(|P_t - \tilde P_t|^2)  \eqdef  \rho(t). 
\]
By the choice of $P_0$ and $\tilde P_0$ as in  \eqref{InitalLaws},  one has  $\rho(0)= \mathcal{W}_2^2(F_0, \tilde F_0)$. Hence, with $\widetilde  C(F_s, \tilde F_s)$ defined as in \eqref{DefCFsTildeFs},  it suffices to show that 
\begin{equation} \label{LocalEstRho}
\rho(t) \leq \rho(0) + \int_0^t \widetilde  C(F_s, \tilde F_s)  \Psi( \rho(s)) \ud s. 
\end{equation}
Since $R_s$ has marginals $F_s$ and $\tilde F_s$, the coupled SDEs \eqref{SDE1} and \eqref{SDE2} can be rewritten as 
\begin{equation*}
%\label{SDE1}
P_t = P_0 + \int_0^t \int_{\R^3 \times \R^3} \Sigma(P_s, p) \,  W(\ud p, \ud \tilde p, \ud s) 
%\\
+ \int_0^t \int_{\R^3 \times \R^3} B(P_s, p) R_s ( \ud p,  \ud \tilde p)  \ud s, 
\end{equation*}
\begin{equation*}
%\label{SDE2}
\tilde P_t = \tilde P_0 + \int_0^t \int_{\R^3 \times \R^3} \Sigma(\tilde P_s, \tilde p) \,  W(\ud p, \ud \tilde p, \ud s) 
%\\
+ \int_0^t \int_{\R^3 \times \R^3} B( \tilde P_s, \tilde p) R_s ( \ud p, \ud \tilde p) \ud s.  
\end{equation*}
Applying It\^o's formula and then taking expectations, one obtains 
\begin{multline*}
\mathbb{E}[|P_t -\tilde P_t|^2]   = \mathbb{E}[|P_0 - \tilde P_0|^2]  
\\
 + \sum_{i, j =1}^3 \int_0^t \int_{\R^3 \times \R^3 } \mathbb{E} \left[ (\Sigma^{ij}(P_s, p) - \Sigma^{ij} (\tilde P_s, \tilde p) )^2 \right] R_s (\ud p, \ud \tilde p) \ud s 
 \\
 + 2 \int_0^t \int_{\R^3 \times \R^3 } \mathbb{E} \left[ (B(P_s, p ) - B(\tilde P_s, \tilde p)) \cdot (P_s - \tilde P_s ) \right] R_s (\ud p, \ud \tilde p) \ud s 
 \\
 = \rho(0) + \int_0^t \mathcal{A}_s \ud s + 2 \int_0^t \mathcal{B}_s \ud s. 
\end{multline*}
Here $\mathcal{A}_s$ and $\mathcal{B}_s$ are defined as the second and third terms in the integrals in the previous step.

Denote by  $Q_s( \ud p, \ud \tilde p)$  the coupled law of the pair $(P_s, \tilde P_s)$. Applying Proposition \ref{EstDoubleB} and the fact that $R_s, Q_s \in \ProbJ(F_s, \tilde F_s)$ we then have
\begin{multline*}
 |\mathcal{B}_s|   \leq \int_{\R^3 \times \R^3} \int_{\R^3 \times \R^3} |p - \tilde p| |B(p, q) - B(\tilde p, \tilde q) | Q_s(\ud p, \ud \tilde p) R_s( \ud q, \ud  \tilde q)   
 \\
  \leq  \widetilde  C(F_s, \tilde F_s)  \Psi\left( \int_{\R^3 \times \R^3} |p - \tilde p|^2 Q_s(\ud p, \ud \tilde p) \right) 
  \\
  + C(F_s, \tilde F_s) 
  \Psi\left( \int_{\R^3 \times \R^3} |q - \tilde q|^2 R_s(\ud q, \ud \tilde q) \right), 
\end{multline*}
where 
\[
\widetilde C(F_s, \tilde F_s) = c \left( \|  F_s \|_{L_3^\infty \cap L^1_3} +  \|  \tilde F_s \|_{L_3^\infty \cap L^1_3}+1 \right), 
\]
and $c>0$ here is a universal constant. 

By the estimate \eqref{EstSigIntegral} in Proposition \ref{PropEstIntegralSingle}, we have
\begin{multline*} 
\mathcal{A}_s  
 = \int_{\R^3 \times \R^3} \int_{\R^3 \times \R^3} |\Sigma(p, q) - \Sigma(\tilde p, \tilde  q)|^2 Q_s(\ud p, \ud \tilde p) R_s(\ud q, \ud \tilde q)  
 \\
 \leq 2  \int_{\R^3 \times \R^3} \int_{\R^3 \times \R^3}  |\Sigma(p, q) - \Sigma(\tilde p,  q)|^2 Q_s(\ud p, \ud \tilde p) R_s(\ud q, \ud \tilde q)   
 \\
 \quad +  2  \int_{\R^3 \times \R^3} \int_{\R^3 \times \R^3}  |\Sigma( \tilde p, q) - \Sigma(\tilde p,  \tilde q)|^2 Q_s(\ud p, \ud \tilde p) R_s(\ud q, \ud \tilde q) 
 \\
  = 2 \int_{\R^3 \times \R^3} Q_s (\ud p, \ud \tilde p) \int_{\R^3 \times \R^3}  |\Sigma(p, q) - \Sigma(\tilde p,  q)|^2 F_s(q) \ud q  
  \\
 \quad + 2  \int_{\R^3 \times \R^3} R_s (\ud q, \ud \tilde q) \int_{\R^3 \times \R^3}  |\Sigma( \tilde p, q) - \Sigma(\tilde p,  \tilde q)|^2 \tilde F_s(\tilde p) \ud \tilde p  
 \\
 \leq C(F_s)  \int_{\R^3 \times \R^3} \Psi(|p - \tilde p|^2) Q_s (\ud p, \ud \tilde p) 
 \\
 + C( \tilde F_s) \int_{\R^3 \times \R^3} \Psi(|q- \tilde q|^2) R_s (\ud q, \ud \tilde q),
\end{multline*} 
where we recall that 
\[
C(F_s) = c \|  F_s(q) \|_{L^\infty_7 \cap L^1_7},  \quad C(\tilde F_s) = c \| \tilde F_s(q) \|_{L^\infty_7 \cap L^1_7}.  
\]
Using Jensen's inequality \eqref{Jensen}, we have that
\begin{multline*} 
\mathcal{A}_s \leq \widetilde C(F_s, \tilde F_s)   \Psi\left( \int_{\R^3 \times \R^3} |p - \tilde p|^2 Q_s(\ud p, \ud \tilde p) \right) 
\\
+ C(F_s, \tilde F_s)\Psi\left( \int_{\R^3 \times \R^3} |q - \tilde q|^2 R_s(\ud q, \ud \tilde q) \right), 
\end{multline*} 
but now $\widetilde C(F_s, \tilde F_s)$ is given in \eqref{DefCFsTildeFs}.

Recalling  the definition of $\rho(s)$, we have 
\[
 \mathcal{W}_2^2(F_s, \tilde F_s) = \int_{\R^3 \times \R^3 } |q - \tilde q|^2 R_s( \ud q, \ud \tilde q)  \leq  \int_{\R^3 \times \R^3 } |q - \tilde q|^2 Q_s( \ud q, \ud \tilde q) = \rho(s). 
\]
Since $\Psi$ is increasing, we finally obtain \eqref{LocalEstRho}. We remark that as long as $F_s$ and $\tilde F_s$ have finite 2nd moments, then $\rho(s)$ is bounded. Indeed, 
\[
\rho(s) =  \int_{\R^3 \times \R^3 } |q - \tilde q|^2 Q_s( \ud q, \ud \tilde q) \leq  2 \int_{\mathbb{R}^3} |q|^2 F_s(\ud q) +  2 \int_{\R^3} |\tilde q |^2 \tilde F_s (\ud \tilde q)  < \infty. 
\]
This completes the proof. 
\end{proof}

We will now proceed to prove Proposition \ref{PropUniquenessSDEs}.

\begin{proof}[Proof of Proposition \ref{PropUniquenessSDEs}]
We only need to check the results for  the SDE  \eqref{SDE1}, while \eqref{SDE2} can be treated similarly. We follow the standard scheme of proof as in \cite{MR2718931}. \\

\noindent {\em Step 1.} For fixed $x_0 \in \R^3$ and a prior known $\R^3-$valued progressively measurable process $X= (X_t)_{t \in [0, T]}$, we define the $\R^3-$valued progressively measurable process $(\Phi(x_0, X)_t)_{t \in [0, T]}$ as 
\begin{equation*} 
\Phi(x_0, X)_t = x_0 + \int_0^t \int_{\R^3 \times \R^3} \Sigma(X_s, p) W(\ud p, \ud \tilde p, \ud s) 
%\\
+ \int_0^t \int_{\R^3} B(X_s, p) F_s(p)  \ud p \ud s. 
\end{equation*} 
We claim that $(\Phi(x_0, X)_t)_{t \in [0, T]}$ is a.s. continuous and 
\begin{multline}  \label{MeanSquareEst}
\mathbb{E}\left[ \sup_{t \in [0, T]} |\Phi(x_0, X)_t|^2 \right] \lesssim
 |x_0|^2 
 + \int_0^T C(F_s) \ud s + \left(\int_0^T C(F_s) \ud s \right)^2 < \infty,
\end{multline}
where $C(F_s) = c \| (q^0)^3 F_s(q)\|_{L^\infty_q \cap L^1_q}$ is uniformly bounded in time $s \in [0, T]$. 
Indeed, following the proof of \eqref{TriSig} in Proposition \ref{PropSig}, we obtain 
\[
|\Sigma(p, q)|^2 \leq c \left( (q^0)^3 |p-q|^{-1} \IN_A(p, q) + (q^0)^2 \IN_{A^c}(p, q) \right). 
\]
Therefore a.s. we have
\begin{equation}\label{locSigMoments}
\begin{split}
 &  \int_{\R^3 } |\Sigma(X_s, p)|^2 R_s( \ud p, \ud \tilde p) \leq \sup_{p \in \R^3} \int_{\R^3} |\Sigma(p, q)| F_s( \ud q)   \\
 &  \leq \sup_{p \in \mathbb{R}^3} \bigg[ \int_{\R^3} |p-q|^{-1} \left(  (q^0)^3 F_s (q)  \right) \ud q  + \int_{\R^3} (q^0)^2 F_s (q) \ud q   \bigg] 
 \\
 & \leq C(F_s)  < \infty,
\end{split}   
\end{equation}
noting that  the first marginal of $R_s$ is $F_s$ and using \eqref{LemIntI} with $\alpha = -1$. Similarly, using \eqref{LemIntI} with $\alpha = -2$ and following the proof of Proposition \ref{PropEstB}, we obtain
\[
|B(p, q) | \leq c \left( q^0 |p-q|^{-2} \IN_A(p, q) + \IN_{A^c} (p, q) \right),  
\]
we obtain a.s. 
\begin{equation}\label{locBMoments}
\int_{\R^3} |B(X_s, p)|F_s(p) \ud p  \leq \sup_{p \in \R^3 }  \int_{\R^3} |B(p, q) | F_s (q) \ud q \leq C(F_s) < \infty. 
\end{equation}
Combining \eqref{locSigMoments} and \eqref{locBMoments}, we can obtain that $(\Phi(x_0, X)_t)_{t \in [0, T]}$ is a.s. continuous by Kolmogorov's continuity theorem. The above mean square estimate \eqref{MeanSquareEst} can be easily deduced from the Doob's martingale inequality, in particular we have that
\[
\mathbb{E}[\sup_{ t \in [0, T]} |\Phi(x_0, X)_t|^2] \leq 4 \mathbb{E} [|(\Phi(x_0, X)_T|^2 ]. 
\]

\smallskip 
\noindent {\em Step 2.} We now prove the uniqueness of the process $((\Phi(x_0, X)_t)_{t \in [0, T]}$ given $(X_t)_{t \in [0, T]}$. Let $(X_t)_{t \in [0, T]}$ and $(Y_t)_{t \in [0, T]}$ be two progressively measurable processes, we want to establish that 
\begin{multline}\label{Step2Eq}
 \Delta_t   \eqdef  \mathbb{E}[|\Phi(x_0, X)_t - \Phi(x_0, Y)_t|^2]  
 \\
 \leq   \int_0^t C(F_s) \Big\{  \Psi\Big( \mathbb{E}[ |\Phi(x_0, X)_s - \Phi(x_0, Y)_s |^2]\Big) + \Psi\Big(  \mathbb{E}[|X_s - Y_s|^2]\Big) \Big\},
\end{multline}
where 
$
C(F_s) = c \,  \| F_s  \|_{L^\infty_7 \cap L^1_7}. 
$

Indeed, using It\^o's formula and then taking expectation, 
\begin{multline*}
\Delta_t   = 
\sum_{i, j=1}^3 \int_0^t \int_{\R^3 \times \R^3} \mathbb{E} \left[ (\Sigma^{ij}(X_s, p) - \Sigma^{ij}(Y_s, p))^2 \right] R_s( \ud p, \ud \tilde p) \ud s 
\\
 + 2 \int_0^t \int_{\R^3 \times \R^3} \mathbb{E} \left[ \Big(B(X_s, p) - B(Y_s, p) \Big) \cdot \Big(\Phi(x_0, X)_s - \Phi(x_0, Y)_s \Big) \right] 
 \\
 \times F_s(p) \ud p \ud s.  
\end{multline*}
Since $R_s$ has the first marginal $F_s$, applying Proposition \ref{PropEstIntegralSingle}, we have 
\[
\begin{split}
\Delta_t & \leq c \int_0^t  \mathbb{E} \bigg[ \int_{\R^3} |\Sigma(X_s, p) - \Sigma(Y_s, p)|^2 F_s(p) \ud p  \bigg] \ud s \\
&  + c \int_0^t \mathbb{E} \bigg[ \Big|\Phi(x_0, X)_s - \Phi(x_0, Y)_s\Big| \cdot \int_{\R^3} |B(X_s, p) - B(Y_s, p)| F_s(p) \ud p  \bigg] \ud s \\
& \leq  \int_0^t C(F_s) \mathbb{E} \Big[  \Psi(|X_s - Y_s|^2) + |\Phi(x_0, X)_s - \Phi(x_0, Y)_s| \Psi(|X_s - Y_s|) \Big] \ud s. 
\end{split}
\]
Since $\Psi$ is increasing and $x \Psi (x) \leq \Psi(x^2)$ for any $x \geq 0$, one has for any $x, y \geq 0$, 
\[
x \Psi(y) \leq \IN_{ \{ x \leq y\} } y \Psi(y) + \IN_{ \{ x \geq y \} } x \Psi(x)  \leq \Psi(x^2) + \Psi(y^2). 
\]
We thus have 
\[
\Delta_t \leq  \int_0^t C(F_s) \mathbb{E} \Big[  \Psi(|X_s - Y_s|^2) + \Psi\left( |\Phi(x_0, X)_s - \Phi(x_0, Y)_s|^2 \right) \Big] \ud s.
\]
Applying Jensen's inequality \eqref{Jensen} to two terms  $\mathbb{E}[\Psi(\cdot )]$, we finally obtain the inequality in \eqref{Step2Eq}. 

\smallskip

\noindent {\em Step 3. } We now check the uniqueness of \eqref{SDE1}. Consider two solutions $P = \Phi(P_0, P)$ and $\tilde P = \Phi(P_0, \tilde P)$ and let $\rho(t) = \mathbb{E}(|P_t - \tilde P_t |^2 )$. By Step 1, $\rho(t)$ is bounded on $[0, T]$. By Step 2, one has 
\[
\rho(t) \leq \int_0^t \gamma(s) \Psi(\rho(s)) \ud s, 
\]
where $\gamma(s) =  C(F_s) \in L^\infty([0, T])$. Lemma \ref{GronwallPsi} yields that $\rho(t) =0$. This means for any $t \in [0, T]$,  a.s. $P_t = \tilde P_t$. The continuity of $(\Phi(x_0, X)_t)_{t \in [0, T]}$ forces that a.s. $(P_t)_{t \in [0, T]} = (\tilde P_t)_{t \in [0, T]}$.   

\smallskip 
\noindent {\em Step 4.}  We now prove the existence of a solution to \eqref{SDE1} using Picard iteration. Define  $P^0$ by $P^0_t \equiv P_0$ and then by induction $P^{n+1}_t  = \Phi(P_0, P^n_t) $. Set $\rho_{n, k}(t) = \sup_{s \in [0, t]} \mathbb{E}[|P_s^{n+k} - P_s^n|^2]$. Again by Step 1, we have that $\sup_{n, k} \|\rho_{n, k}(t) \|_{L^\infty[0, T]} < \infty$. 
Then by Step 2, 
\[
\rho_{n+1, k}(t) \leq \int_0^t \gamma(s) \left[ \Psi(\rho_{n+1, k}(s)) + \Psi(\rho_{n, k}(s) ) \right] \ud s, 
\]
where $\gamma(s) = C(F_s) \in L^\infty([0, T])$.  We define $\rho_n(t)  \eqdef  \sup_k \rho_{n, k} (t)$. Since  $\Psi$ is increasing, 
\[
\rho_{n+1}(t) \leq \int_0^t \gamma(s) \left[ \Psi(\rho_{n+1}(s)) + \Psi(\rho_{n}(s) ) \right] \ud s. 
\]
Finally, set $\rho(t)  \eqdef  \limsup_n \rho_n(t)$. By (reverse) Fatou's lemma, 
one has 
\[
\rho(t) \leq 2 \int_0^t \gamma(s) \Psi(\rho(s)) \ud s. 
\]
Then Lemma \ref{GronwallPsi} guarantees that $\rho(t) \equiv 0$ for all $t \in [0, T]$. We obtain
\[
\limsup_n \sup_k \sup_{t \in [0, T]} \mathbb{E} [|P_t^{n+k} - P_t^n|^2] =0. 
\]
This means that the sequence $(P_t^n)_{t \in [0, T]}$ is a Cauchy sequence in the space $L^\infty([0, T], L^2(\Omega))$. Thus there exists a process $(P_t)_{t \in [0, T]}$ such that 
\[
\lim_n \sup_{t \in [0, T]} \mathbb{E} [|P_t - P_t^n |^2] =0. 
\]
To conclude this step, it suffices to prove that 
\[
\kappa_n(t)  \eqdef  \mathbb{E}[|\Phi(P_0, P^n)_t - \Phi(P_0, P)_t|^2] \to 0, \quad 
\]
as $n \to \infty$. Combining with  $(P^{n+1}_t)_{t \in [0, T]} = (\Phi(P_0, P^n)_t)_{t \in [0, T]}$ converges to $(P_t)_{t \in [0, T]}$ in $L^\infty([0, T], L^2(\Omega))$, we can thus conclude $P = \Phi(P_0, P)$. The a.s. continuity of $P$ will follow from Step 1. 

We define $\eps_n \eqdef  \sup_{t \in [0, T]} \mathbb{E}[\Psi(|P_t^n - P_t|^2 )]$, which tends to $0$ using Jensen's inequality \eqref{Jensen}. Applying the estimate \eqref{Step2Eq} in Step 2 again, one obtains, for $t \in [0, T]$, 
\[
\kappa_n(t) \leq \int_0^t \gamma(s) \left( \eps_n + \Psi(\kappa_n(s) ) \right)  \ud s, 
\] 
where again $\gamma(s) = C(F_s) \in L^\infty([0, T])$. Similarly, we consider $\kappa(t) =\limsup_n \kappa_n(t).$ Using (reverse) Fatou's lemma again, we obtain 
\[
\kappa(t) \leq \int_0^t \gamma(s) \Psi(\kappa(s)) \ud s. 
\]
Then Lemma \ref{GronwallPsi} yields that $\kappa(t) \equiv 0$ for any $t \in [0, T]$. We have proved that $\lim_n \kappa_n(t) =0$, which concludes this step.

\smallskip
\noindent {\em Step 5. } It remains to check that for any $t \in [0, T]$, $\mbox{Law}(P_t) = F_t$, where $P=(P_t)_{t \in [0, T]}$ is the unique solution to \eqref{SDE1} or 
\begin{equation*}
%\label{SDE1}
P_t = P_0 + \int_0^t \int_{\R^3 \times \R^3} \Sigma(P_s, p) \,  W(\ud p, \ud \tilde p, \ud s) + \int_0^t \int_{\R^3} B(P_s, p) F_s (p) \ud p \ud s. 
\end{equation*}
Set $G_s = \mbox{Law}(P_s)$ for any $s \in [0, T]$. We note that $(G_t)_{t \in [0, T]}$ solves the linear relativistic Landau equation, i.e. for any test function $\varphi \in C_b^2(\R^3)$, 
\begin{equation}\label{LinearLandau}
\int_{\R^3} \varphi(p) G_t(\ud p) = \int_{\R^3} \varphi(p) F_0(\ud p) 
%\\
+ \int_0^t \int_{\R^3 \times \R^3} L \varphi(p, q) G_s (\ud p) F_s(q) \ud q \ud s,   
\end{equation}
where $L$ is defined in \eqref{LDef}. Indeed, we apply It\^o's formula to $\ud \varphi(P_t)$, 
\[
\begin{split}
\varphi(P_t) & = \varphi(P_0) + \int_0^t \int_{\R^3 \times \R^3} \sum_{i, j=1}^3 \partial_i \varphi(P_s) \Sigma^{ij}(P_s, p) W_j(\ud p, \ud \tilde p, \ud s)    \\
& + \int_0^t \int_{\R^3} \sum_{i=1}^3 \partial_i \varphi(P_s) B_i(P_s, p) F_s(p) \ud p \ud s \\
& + \frac{1}{2} \int_0^t \int_{\R^3 \times \R^3} \sum_{i, j, k=1}^3 \partial_{ij } \varphi(P_s) \Sigma^{ik}(P_s, p) \Sigma^{jk} (P_s, p) R_s(\ud p, \ud \tilde p) \ud s. 
 \end{split}
\]
Taking expectations (which makes the first integral vanish), and noting that $\mbox{Law}(P_s) = G_s$,  $\mbox{Law}(P_0) = F_0$ and $\Phi= \Sigma\,  \Sigma^\top$ or $\Phi^{ij} = \sum_k \Sigma^{ik}\Sigma^{jk}$, one reaches the conclusion. 

For the moment we will assume that we have uniqueness for the linear Landau equation. Since $(F_t)_{t \in [0, T]}$ is a weak solution to the relativistic Landau \eqref{RelLandau}, it is of course a weak solution to the linear Landau \eqref{LinearLandau}. By the uniqueness of the linear Landau, one concludes that $F_t = G_t$ for all $t \in [0, T]$. 

It remains to show the uniqueness  for the linear Landau \eqref{LinearLandau}. We apply Theorem B.1 in Horowitz-Karandikar \cite{MR1042343}, see also Theorem 5.2 in Bath-Karandikar \cite{BK93}. Consider, for $t \in [0, T]$, $p \in \R^3$ and $\varphi \in C_b^2$, the following operator 
\[
\mathcal{A}_t \varphi(p)  \eqdef  \int_{\R^3} L \varphi(p, q) F_t( q) \ud q. 
\]
A stochastic process $(X_t)_{t \in [t_0, T]}$ is said to solve the martingale problem for $(C_b^2, \mathcal{A}_t)$ if for all $\varphi \in C_b^2$, the process $\varphi(X_t) - \int_{t_0}^t \mathcal{A}_s \varphi(X_s) \ud s$, defined for  $t \in [t_0, T]$, is a martingale. To apply Theorem B.1 in \cite{MR1042343}, we only need to check that 
\begin{enumerate}
\item[(i)] there is a countable family $(\varphi_k)_{k \geq 1} \subset C_b^2$ such that for all $t \in [0, T]$, $\{ (\varphi_k, \mathcal{A}_t \varphi_k)\}_{ k \geq 1}$  is dense in $\{ (\varphi, \mathcal{A}_t \varphi) \vert  \varphi \in C_b^2  \}$ with respect to bounded-pointwise convergence; 

\item[(ii)] for any $(t_0, x_0)$ in $[0, T] \times \mathbb{R}^3$, there exists a unique (in law) solution $(X_t)_{t \in [t_0, T]}$ to the martingale problem for $(C_b^2, \mathcal{A}_t)$ such that $X_{t_0} = x_0$. 
\end{enumerate}
We now check these two points. First choose a countable family of functions $(\varphi_k)_{k \geq 1} \subset C_b^2$, dense in $C_b^2$, endowed with the norm $\|\varphi \|_{C_b^2}  \eqdef  \| \varphi\|_{L^\infty} + \|D \varphi \|_{L^\infty} + \|D^2 \varphi \|_{L^\infty}$. Note that 
\[
\begin{split}
 & |\mathcal{A}_t \varphi(p)| \leq \int_{\mathbb{R}^3} |L \varphi(p, q)| F_t(q) \ud q  \\
 & \leq \frac{1}{2} \sum_{i,j=1}^3 \int_{\R^3} |\Phi^{ij} (p, q)| |\partial_{ij}^2 \varphi(p)| F_t(q) \ud q + \sum_{i=1}^3 \int_{\R^3} |B_i(p, q) | |\partial_i \varphi(p) | F_t(q) \ud q. 
 \end{split}
\]
Recall Lemma \ref{lem:EstOfPhi}  and Lemma \ref{lem:BTrivial}, in particular we have
\[
|\Phi(p, q) | \leq c\left( 1 + q^0 |p-q|^{-1} \right), \quad |B(p, q)| \leq c \left( 1+ q^0 |p-q|^{-2}   \right).  
\]
Hence
\[
|\mathcal{A}_t \varphi(p)| \leq c \| \varphi \|_{C_b^2} \left(  1 + \|  F_t \|_{L^\infty_1 \cap L^1_1} \right)
\] 
by \eqref{LemIntI} with $\alpha=-1$ and $\alpha=-2$. This implies that $\mathcal{A}_t \varphi_k$ converges uniformly (stronger than  the bounded-pointwise convergence) to a certain $\mathcal{A}_t \varphi$. 

To prove (ii), we observe that the martingale problem for $(C_b^2, \mathcal{A}_t)$ with $X_{t_0} = x_0$ corresponds to the following SDE
\[
X_t = x_0 + \int_{t_0}^t \int_{\R^3 \times \R^3} \Sigma(X_s, p) W(\ud p, \ud \tilde p, \ud s) + \int_{t_0}^t \int_{\R^3} B(X_s, p) F_s(\ud p) \ud s. 
\]
From the previous Step 1 to Step 4, we have proved the strong existence and uniqueness but only in the case $t_0=0$ and $x_0 = P_0$. The generalization to the above case is straightforward. We have  thus proved point (ii). 

This completes the proof of Proposition \ref{PropUniquenessSDEs}. 
\end{proof}

\section{Proof of the main theorem}\label{proof.main.sec}

We now prove our Main Theorem \ref{MainTheorem}.  The proof will make crucial use of Proposition \ref{PropIntegral} and Lemma \ref{GronwallPsi}. 

\begin{proof}[Proof of Theorem \ref{MainTheorem}] (i) Assume that  $(F_t)_{t \in [0, T]}$ and $(\tilde F_t)_{t \in [0, T]}$ are two weak solutions  to the relativistic Landau Eq. \eqref{RelLandau} with 
\[
\int_0^T   \|  F_s\|_{L^\infty_7 \cap L^1_7} \ud s < \infty,  
\quad 
\int_0^T   \|  \tilde F_s\|_{L^\infty_7 \cap L^1_7} \ud s < \infty. 
\]
Note that the existence of weak solutions is proven in \cite{StrainTas}; however \cite{StrainTas} does not obtain $L^\infty$ bounds on the solution.

By Proposition \ref{PropIntegral},  there exists  a bounded function $\rho: [0, T]\mapsto [0, \infty)$, such that for any $t \in [0, T]$, 
\[
\mathcal{W}_2^2 (F_t, \tilde F_t) \leq \rho(t), \quad \rho(t) \leq \mathcal{W}_2^2(F_0, \tilde F_0) +  \int_0^t \gamma(s)   \Psi(\rho(s)) \ud s, 
\]
where   
\begin{equation*} 
\gamma(s) = C(F_s, \tilde F_s)= c \left(  \|  F_s(q)\|_{L^\infty_7 \cap L^1_7} +  \|  \tilde  F_s(q)\|_{L^\infty_7 \cap L^1_7}  \right) \in L^1([0, T]).  
\end{equation*}
By Lemma \ref{GronwallPsi}, if initially $\rho(0) = \mathcal{W}_2^2(F_0, \tilde F_0) =0$, then for any $t \in [0, T]$, $\rho(t)=0$ and thus $\mathcal{W}_2(F_t, \tilde F_t) =0$. Thus $(F_t)_{t \in [0, T]} = (\tilde F_t)_{t \in [0, T]}. $

(ii) Consider a family of  weak solutions $(F_t)_{t \in [0, T]}$ and $(F_t^n)_{t \in [0, T]}$ to \eqref{RelLandau} such that 
\[
\sup_n \int_0^T    \left( \|  F_s\|_{L^\infty_7 \cap L^1_7} +   \|  F_s^n \|_{L^\infty_7 \cap L^1_7}  \right) \ud s < \infty, 
\]
and $\rho_n(0)  \eqdef  \mathcal{W}_2^2(F_0, F_0^n) \to 0$ as $n \to 0$. Then applying Proposition \ref{PropIntegral} again, one has a family of bounded functions $\rho_n: [0, T] \mapsto [0, \infty)$, such that 
\[
\mathcal{W}_2^2 (F_t, F_t^n) \leq \rho_n(t), \quad \rho_n(t) \leq \rho_n(0) + \int_0^t \gamma_n(s) \Psi(\rho_n(s)) \ud s. 
\]
Lemma \ref{GronwallPsi} part ii) implies that 
\[
\int_{\rho_n(0)}^{\rho_n(t) } \frac{1}{\Psi(y)} \ud y \leq \int_0^t \gamma_n(s) \ud s, \quad \mbox{for any}\  t \in [0, T]. 
\]
Since for any $\eps> 0$, $\int_0^\eps 1/\Psi(y) \ud y = + \infty$ and $ \sup_n \int_0^T \gamma_n(s) \ud s < \infty$ and $\rho_n(0) \to 0$ as $n \to \infty$, one finally obtains that $\lim_n \sup_{t \in [0, T]}  \rho_n(t)  =0$ and then consequently we have that $\lim_n \sup_{t \in [0, T]} \mathcal{W}_2^2(F_t, F_t^n) =0. $
 \end{proof}

%%%%%%%%%%%%%%%%%%%%%%%%%%%%%%%%%%%%%%%%%%%%%%%%%%%%%%%%%%%%%%%%%%%%%%%%%%%%%%%%%% 

%%% copy and paste the BBL file below here

\providecommand{\MR}[1]{}
\providecommand{\bysame}{\leavevmode\hbox to3em{\hrulefill}\thinspace}
\providecommand{\MR}{\relax\ifhmode\unskip\space\fi MR }
\providecommand{\MRhref}[2]{%
  \href{http://www.ams.org/mathscinet-getitem?mr=#1}{#2}
}
\providecommand{\href}[2]{#2}

\end{document}